\documentclass[12pt,leqno,twoside,sort]{amsart}

\usepackage{amssymb,amsmath,amsthm,soul,color}

\usepackage[normalem]{ulem}

\usepackage{mathrsfs}

\usepackage[utf8]{inputenc}

\usepackage{todonotes}
\setuptodonotes{inline}

\usepackage[left=2cm, right=2cm, top=2cm, bottom=2cm]{geometry}

\usepackage{url}

\usepackage[numbers,sort&compress]{natbib}

\usepackage{fixmath}

\usepackage[T1]{fontenc}
\usepackage{lmodern}
\usepackage{microtype}
\usepackage{mathtools}

\usepackage[colorlinks=true,urlcolor=blue,
citecolor=red,linkcolor=blue,linktocpage,pdfpagelabels,
bookmarksnumbered,bookmarksopen]{hyperref}

\usepackage{bbm}

\usepackage{xcolor,cancel}

\newtheorem{Th}{Theorem}[section]
\newtheorem{Prop}[Th]{Proposition}
\newtheorem{Lem}[Th]{Lemma}

\newtheorem{Rem}[Th]{Remark}

\newcommand{\R}{\mathbb{R}}

\newcommand{\Z}{\mathbb{Z}}
\newcommand{\N}{\mathbb{N}}

\newcommand{\cC}{{\mathcal C}}
\newcommand{\cD}{{\mathcal D}}

\newcommand{\cM}{{\mathcal M}}

\newcommand{\Ga}{\Gamma}

\newcommand{\weakto}{\rightharpoonup}

\renewcommand{\div}{\mathrm{div}\,}

\numberwithin{equation}{section}

\newcommand{\supp}{\mathrm{supp}\,}

\newcommand{\loc}{\mathrm{loc}}

\newcommand{\vertiii}[1]{{\left\vert\kern-0.25ex\left\vert\kern-0.25ex\left\vert #1 
    \right\vert\kern-0.25ex\right\vert\kern-0.25ex\right\vert}}

\begin{document}

\title[Nonlinear scalar field equations with a critical Hardy potential]{Nonlinear scalar field equations with a critical Hardy potential}

\author[B. Bieganowski]{Bartosz Bieganowski}
\author[D. Strzelecki]{Daniel Strzelecki}

\address[B. Bieganowski]{\newline\indent
	Faculty of Mathematics, Informatics and Mechanics, \newline\indent
	University of Warsaw, \newline\indent
	ul. Banacha 2, 02-097 Warsaw, Poland}
\email{\href{mailto:bartoszb@mimuw.edu.pl}{bartoszb@mimuw.edu.pl}}

\address[D. Strzelecki]{\newline\indent
	Faculty of Mathematics, Informatics and Mechanics, \newline\indent
	University of Warsaw, \newline\indent
	ul. Banacha 2, 02-097 Warsaw, Poland
    \newline\indent
    and
    \newline\indent
    Faculty of Mathematics and Computer Science
    \newline\indent
    Nicolaus Copernicus University in Toruń
    \newline\indent
    ul. Chopina 12/18, 87-100 Toruń, Poland
    }
\email{\href{mailto:dstrzelecki@mimuw.edu.pl}{dstrzelecki@mimuw.edu.pl}}

\date{\today} 

\begin{abstract} 
We study the existence of solutions for the nonlinear scalar field equation
$$-\Delta u - \frac{(N-2)^2}{4|x|^2} u  = g(u),  \quad \mbox{in } \R^N \setminus \{0\},$$
where the potential $-\frac{(N-2)^2}{4|x|^2}$ is the critical Hardy potential and $N \geq 3$. The nonlinearity $g$ is continuous and satisfies general subcritical growth assumptions of the Berestycki-Lions type. The problem is approached using variational methods within a non-standard functional setting. The natural energy functional associated with the equation is defined on the space $X^1(\R^N)$, which is the completion of $H^1(\R^N)$ with respect to the norm induced by the quadratic part of the functional. We show that every nontrivial solution satisfies the Poho\v{z}aev constraint $\cM$ and establish the existence of a nontrivial solution $u_0 \in X^1(\R^N)$ that minimizes the energy functional on $\cM$. Furthermore, assuming $g$ is odd, we prove the existence of at least one non-radial solution.

\medskip

   
\end{abstract}

\maketitle

\section{Introduction}

We study the existence and qualitative properties of solutions for the nonlinear scalar field equation
\begin{equation}\label{eq:intro_main}
-\Delta u - \mu \frac{u}{|x|^2} = g(u), \quad \mbox{in } \R^N \setminus \{0\},
\end{equation}
where $N \geq 3$. This paper focuses specifically on the \textit{critical case} $\mu = \mu_c = \frac{(N-2)^2}{4}$, where $\mu_c$ is the best constant in the classical Hardy inequality. The inverse-square potential $\frac{1}{|x|^2}$, commonly referred to as the Hardy potential, is of significant importance in quantum mechanics, modeling diverse phenomena such as the interaction of a charge with a dipole or certain molecular potentials \cite{Alhaidari2014, Bawin, Case}. The nonlinearity $g$ is assumed to satisfy the general Berestycki-Lions conditions, (G1)--(G4) below, which do not include the celebrated Ambrosetti-Rabinowitz condition.

This problem presents a combination of two significant analytical challenges: the \textit{criticality of the potential} and the \textit{generality of the nonlinearity}.

The foundational case, where $\mu = 0$, was famously studied by Berestycki and Lions in their seminal work \cite{BL}. They established the existence of a ground state solution in the standard Sobolev space $H^1(\R^N)$ under these general conditions. The inclusion of the \textit{critical} Hardy potential $\mu_c / |x|^2$, however, fundamentally alters the problem's functional analytic structure. The associated quadratic form $\xi(u) = \int_{\R^N} (|\nabla u|^2 - \mu_c |x|^{-2} u^2) \, dx$ is merely positive semi-definite, not coercive, on $H^1(\R^N)$. This loss of coercivity necessitates moving the analysis to a larger, Sobolev-type space, $X^1(\R^N)$, which is the completion of $H^1(\R^N)$ with respect to the problem's natural norm. The properties of this space, which are essential for our analysis, are based on foundational work by authors such as Frank \cite{Frank}, Suzuki \cite{Suzuki}, Mukherjee, Nam, and Nguyen \cite{Mukherjee}, Trachanas and Zographopoulos \cite{TrachanasZographopoulos}.

Our focus on the \textit{critical} case $\mu = \mu_c$ should be contrasted with the \textit{subcritical} case $\mu < \mu_c$. For subcritical potentials, the quadratic form remains coercive (or can be made so with an equivalent norm), allowing for a more standard application of variational methods within $H^1(\R^N)$ or related weighted spaces. The subcritical case was of interest to many authors (see e.g. \cite{GuoMederski, LiLiTang, MR3989664} and references therein). For instance, Li, Li, and Tang \cite{LiLiTang} studied the subcritical problem under Berestycki-Lions-type conditions, establishing the existence of a ground state solution \textit{within} the $H^1(\R^N)$ space. We find that,  under slightly more restrictive assumptions, for the critical potential, any nonnegative and nontrivial solution is genuinely singular and \textit{does not} belong to $H^1(\R^N)$. This highlights a fundamental difference in the qualitative nature of solutions as the potential transitions from subcritical to critical. Other related works, such as those by Deng, Jin, and Peng \cite{Deng}, Felli, Pistoia \cite{Felli}, Smets \cite{Smets}, and Terracini \cite{Terracini96}, have investigated problems combining the Hardy potential with the \textit{critical Sobolev exponent} $2^* = \frac{2N}{N-2}$, which introduces a different set of compactness challenges related to the nonlinearity itself.

The second major challenge, the absence of the Ambrosetti-Rabinowitz condition, means that standard mountain pass arguments do not guarantee the boundedness of Palais-Smale sequences. For the classical problem ($\mu=0$), this difficulty is often overcome by seeking solutions on the Poho\v{z}aev constraint. A systematic approach to find such solutions was developed by Jeanjean \cite{Jeanjean}, and later refined by Jeanjean, Tanaka \cite{JeanjeanTanaka} and Hirata, Ikoma, Tanaka \cite{HirataIkomaTanaka}. They introduced an \textit{augmented functional} technique to construct a bounded Poho\v{z}aev-Palais-Smale sequence at the mountain pass level. We adopt this strategy to construct a bounded Poho\v{z}aev-Palais-Smale sequence ((PPS)-sequence) for our singular problem.

The boundedness of a (PPS) sequence is not sufficient; one must additionally establish its convergence. The primary tool for this in translation-invariant problems is the profile decomposition method, developed by Gérard \cite{Gerard} and Nawa \cite{Nawa}. This method fails in our setting because the $X^1(\R^N)$ space, due to the potential centered at the origin, is \textit{not} translation-invariant. A central technical contribution of this paper is the development of a \textit{new} profile decomposition specifically adapted to bounded (PPS) sequences in this non-translation-invariant space. This decomposition is crucial as it shows that any potential \textit{bubbles} (weak limits of translated sequences) must be solutions to the limiting problem \textit{without} the potential, i.e., in the classical $H^1(\R^N)$ space. By showing that our mountain pass level is strictly less than the ground state energy of this limiting problem (an energy comparison adapting \cite{JeanjeanTanaka}), we prove that no bubbles can form, forcing the strong convergence of the sequence to our solution.

Finally, we extend our analysis to the existence of non-radial solutions. This is a classical and challenging question, famously posed by Berestycki and Lions \cite{BL} for the $\mu=0$ case. This question was recently answered for the classical problem by Mederski \cite{MR4173560}, who found non-radial solutions by minimizing the functional on a subspace of $H^1(\R^N)$-functions orthogonal to the space of radial functions. We successfully adapt this symmetry-based minimization argument of \cite{MR4173560} from the standard $H^1(\R^N)$ space to our non-translation-invariant, singular $X^1(\R^N)$ setting, thereby establishing the existence of non-radial solutions for the critical Hardy potential problem under general nonlinearities.

In the critical case, \eqref{eq:intro_main} rewrites as
\begin{equation}\label{eq:main}
-\Delta u - \frac{(N-2)^2}{4|x|^2} u  = g(u),  \quad \mbox{in } \R^N \setminus \{0\}.
\end{equation}
We consider the following general conditions
\begin{itemize}
\item[(G1)] $g : \R \rightarrow \R$ is continuous;
\item[(G2)] $-\infty < \liminf_{s \to 0}  \frac{g(s)}{s} \leq \limsup_{s \to 0} \frac{g(s)}{s} = -\ell < 0$;
\item[(G3)] $\lim_{|s| \to \infty} \frac{g(s)}{|s|^{p-1}} = 0$ for some $p \in \left( 2,2^* \right)$;
\item[(G4)] there is $\zeta_0 > 0$ such that $G(\zeta_0) > 0$, where $G(s) := \int_0^{s} g(\tau) \, d\tau$.
\end{itemize}

Assumptions (G1)--(G4) are standard, in the spirit of \cite{BL}, although we do not require oddness of $g$. 

We consider the so-called \textit{positive-mass case} ($\ell > 0$). Therefore, it is convenient to define $f : \R \rightarrow \R$ by
\begin{equation}\label{def:f}
f(s) := \ell s + g(s), \quad s \in \R
\end{equation}
and then $G$ can be rewritten as
$$
G(s) = -\frac{\ell}{2} s^2 + F(s), \quad F(s) := \int_0^s f(\tau) \, d\tau.
$$
Then conditions (G1)--(G4) imply the following
\begin{itemize}
\item[(F1)] $f : \R \rightarrow \R$ is continuous;
\item[(F2)] $-\infty < \liminf_{s \to 0}  \frac{f(s)}{s} \leq \limsup_{s \to 0} \frac{f(s)}{s} = 0 < \ell$;
\item[(F3)] $\lim_{|s| \to \infty} \frac{f(s)}{|s|^{p-1}} = 0$ for some $p \in \left( 2, 2^* \right)$;
\item[(F4)] there is $\zeta_0 > 0$ such that $F(\zeta_0) > \frac{\ell}{2} \zeta_0^2$.
\end{itemize}

The natural energy functional associated with \eqref{eq:main} is given by
\begin{align*}
J (u) &= \frac{1}{2} \int_{\R^N} |\nabla u|^2 - \frac{(N-2)^2}{4} \frac{u^2}{|x|^2} \, dx - \int_{\R^N} G(u) \, dx \\
&=  \frac{1}{2} \int_{\R^N} |\nabla u|^2 - \frac{(N-2)^2}{4} \frac{u^2}{|x|^2} \, dx + \frac{\ell}{2} \int_{\R^N} u^2 \, dx - \int_{\R^N} F(u) \, dx.
\end{align*}
Clearly, $J$ is well defined on $H^1 (\R^N)$, however the quadratic form
$$
u \mapsto \|u\|^2 := \xi(u) +  \int_{\R^N} u^2 \, dx,
$$
where
$$
\xi(u) := \int_{\R^N} |\nabla u|^2 - \frac{(N-2)^2}{4} \frac{u^2}{|x|^2} \, dx,
$$
generates a norm in $H^1 (\R^N)$, which is not complete (and, in particular, is not equivalent to the standard one). Therefore, to work in a Banach space, we need to slightly enlarge the energy space. More precisely, we define
$$
X^1 (\R^N) := \mbox{completion of } H^1(\R^N) \mbox{ with respect to the norm } \| \cdot \|.
$$
For properties of $X^1 (\R^N)$ and the variational setting (including the appropriate definition of $\xi$ and $\|\cdot\|$), see Sections \ref{sect:2} and \ref{sect:3}, respectively. Under the foregoing assumptions, $J$ is a $\cC^1$ functional on $X^1(\R^N)$ and its critical points are weak solutions to \eqref{eq:main}.

It is natural to consider the Poho\v{z}aev constraint associated with \eqref{eq:main}, namely
$$
\cM := \left\{ u \in X^1 (\R^N) \setminus \{0\} \ : \ \xi(u) = 2^* \int_{\R^N} G(u) \, dx \right\}.
$$
In the classical setting it is known that every weak solution of a scalar field equation belongs to the Poho\v{z}aev constraint. 
We show in Proposition \ref{P:Pohozaev} that the same is true in our singular setting, namely every nontrivial critical point of $J$ in $X^1(\R^N)$ belongs to $\cM$, although, in general, it is not of class $H^1$ near the origin. As a consequence, the minimizer of $J$ on $\cM$ found below is a ground state solution.

We are ready to state our first result.

\begin{Th}\label{Th:Main1-Existence}
Suppose that (G1)--(G4) are satisfied. 
\begin{enumerate}
\item[(a)] There exists a nontrivial solution $u_0 \in X^1 (\R^N)$ of \eqref{eq:main} and $J(u_0) > 0$. 
\item[(b)] $u_0$ is a minimizer of the energy functional $J$ on $\cM$, hence it is a ground state solution.
\item[(c)] Every solution $u\in X^1(\R^N)$ of \eqref{eq:main} belongs to $C^{1,\alpha}_{\loc} (\R^N \setminus \{0\})$ for some $\alpha \in (0,1)$.
\item[(d)] If, in addition, $g$ is H\"older continuous, every solution $u\in X^1(\R^N)$ of \eqref{eq:main} is of the class $C^2 (\R^N\setminus \{0\})$.
\item[(e)] If, in addition, $g$ is odd, there exists a nontrivial, nonnegative, radial and radially nonincreasing solution in $X^1 (\R^N) \cap C^2 (\R^N \setminus \{0\})$ of \eqref{eq:main}, which is a minimizer of $J$ on $\cM$.
\item[(f)] Suppose that, in addition, $g(s) \geq 0$ for sufficiently large $s$. Then every nonnegative, nontrivial solution of \eqref{eq:main} in $X^1(\R^N) \cap C^2 (\R^N \setminus\{0\})$ does not belong to $H^1 (\R^N)$.
\end{enumerate}
\end{Th}

In this paper, we present a profile decomposition approach, which is based on the concentration-compactness principle in the spirit of Lions. The decomposition of minimizing sequences, in the spirit of \cite{Nawa, Gerard} is a challenging problem, as the space $X^1 (\R^N)$ is not translation-invariant, and even if the sequence has a weak limit after finding proper translations, it should be a solution to the limiting problem without the singular potential, for which the proper energy space is $H^1 (\R^N)$, not $X^1 (\R^N)$.

Next, we investigate the existence of non-radial solutions. We follow the ideas of \cite{MR4173560, BartschWillem}. Suppose that $N \geq 4$ and write $\R^N = \R^M \times \R^M \times \R^{N-2M}$, where $2 \leq M \leq \frac{N}{2}$. For $x \in \R^N$ we will write $x=(x_1, x_2, x_3)$, where $x_1, x_2 \in \R^M$ and $x_3 \in \R^{N-2M}$. We introduce an action $\tau : \R^N \rightarrow \R^N$ by
$$
\tau(x_1, x_2, x_3) := (x_2, x_1, x_3)
$$
and the subspace
\begin{equation}\label{eq:Xtau}
X_\tau := \left\{ u \in X^1 (\R^N) \ : \ u(x) = - u(\tau x) \mbox{ for a.e. } x \in \R^N \right\}.
\end{equation}
Then, if $u \in X_\tau$ is radially symmetric, then $u \equiv 0$ and therefore $X_\tau$ does not contain nontrivial, radial functions. In fact, $X_\tau \subset H^1 (\R^N)$, see Section \ref{sect:7}.

Let $\mathcal{O} := \mathcal{O}(M) \times \mathcal{O}(M) \times \mathrm{id}_{N-2M} \subset \mathcal{O}(N)$. We introduce the subspace $X^1_\mathcal{O} (\R^N)$ of $\mathcal{O}$-invariant functions.
\begin{Th}\label{Th:Main2}
Suppose that $N\geq 4$, (G1)--(G4) are satisfied and that $g$ is odd. Then there exists a nontrivial solution $v_0 \in \cM \cap X_\tau \cap X^1_{\mathcal{O}} (\R^N)$ of \eqref{eq:main}, which minimizes $J$ on $\cM \cap X_\tau \cap X^1_{\mathcal{O}} (\R^N)$, and
$$
J(v_0) = \inf_{\cM \cap X_\tau \cap X^1_{\mathcal{O}} (\R^N)} J \geq 2 \inf_{\cM} J > 0 .
$$
In particular, $v_0$ is not radial and it is not a ground state solution.
\end{Th}

The paper is organized as follows. Sections \ref{sect:2} and \ref{sect:3} are devoted to the description of the functional and variational setting, respectively. The profile decomposition is proved in Section \ref{sect:4}. Then, the point (a) of Theorem \ref{Th:Main1-Existence}, namely the existence of a nontrivial solution is shown in Section \ref{sect:5}. Section \ref{sect:6} is devoted to proving additional properties of the solution. Section \ref{sect:7} contains the proof of Theorem \ref{Th:Main2}.

In what follows, $\lesssim$ denotes an inequality up to a positive, multiplicative constant and $C$ denotes a generic, positive constant that may vary from one line to another.

\section{Functional setting}\label{sect:2}

Recall the Hardy inequality
\begin{equation} \label{ineq:Hardy}
 \int_{\R^N} \frac{u^2}{|x|^2} \, dx \leq \frac{4}{(N-2)^2} \int_{\R^N} |\nabla u|^2 \, dx, \quad u \in H^1 (\R^N).   
\end{equation}
In particular, it implies that
$$
\| u \|^2 := \int_{\R^N} |\nabla u|^2 - \frac{(N-2)^2}{4} \frac{u^2}{|x|^2} + u^2 \, dx, \quad u \in H^1 (\R^N)
$$
is a norm in $H^1 (\R^N)$. However, due to the optimality of the constant $\frac{(N-2)^2}{4}$ and the fact that optimizers of \eqref{ineq:Hardy} do not belong to $H^1(\R^N)$\footnote{One can construct a function in $X^1 (\R^N)$, which behaves like $|x|^{- \frac{N-2}{2}}$ at the origin. Note that $x \mapsto |x|^{- \frac{N-2}{2}}$ is the optimizer of the Hardy inequality.}, this norm is not complete in $H^1 (\R^N)$. To work in a Hilbert space we will work in the completion of $H^1 (\R^N)$ in this norm. Such a space was already considered in \cite{Suzuki, Mukherjee, TrachanasZographopoulos}.  

Let $X^1 (\R^N)$ denote the completion of $H^1 (\R^N)$ with respect to the norm
\begin{equation}\label{normX1-oryg}
\| u \|^2 = \int_{\R^N} |\nabla u|^2 - \frac{(N-2)^2}{4} \frac{u^2}{|x|^2} + u^2\, dx, \quad u \in H^1 (\R^N).
\end{equation}
Then, from \cite[Theorem 1.2]{Frank} (see also \cite{Suzuki}), the following embeddings are continuous
$$
H^1 (\R^N) \subset X^1 (\R^N) \subset H^s(\R^N)
$$
for any $s \in [0, 1)$, where we set $H^0 (\R^N) = L^2 (\R^N)$. In particular, we have continuous embeddings into Lebesgue spaces
\begin{equation}\label{eq:embeddings}
X^1 (\R^N) \subset L^t (\R^N)
\end{equation}
for $t \in [2, 2^*)$, which are locally compact (the embedding $X^1 (\R^N) \subset L^t_{\loc} (\R^N)$ is compact).

Moreover, see \cite{TrachanasZographopoulos, VazquezZ}, the norm in $X^1 (\mathbb{R}^N)$ has the following characterization
\begin{align}\label{normX1}
\|u\|^2 = \lim_{\varepsilon \to 0+} \left( \int_{|x| > \varepsilon} |\nabla u|^2 \, dx - \frac{(N-2)^2}{4} \int_{|x| > \varepsilon} \frac{u^2}{|x|^2} \, dx - \frac{N-2}{2} \varepsilon^{-1} \int_{|x|=\varepsilon} u^2 \, dS \right) + \int_{\R^N} u^2 \, dx
\end{align}
for $u \in X^1 (\mathbb{R}^N)$, and the scalar product $\langle \cdot, \cdot \rangle$ on $X^1 (\R^N)$ is given by
\begin{align*}
\langle u,v \rangle = \lim_{\varepsilon \to 0+} \left( \int_{|x| > \varepsilon} \nabla u \nabla v \, dx - \frac{(N-2)^2}{4} \int_{|x| > \varepsilon} \frac{u v}{|x|^2} \, dx - \frac{N-2}{2} \varepsilon^{-1} \int_{|x|=\varepsilon} u v \, dS \right) + \int_{\R^N} uv \, dx.
\end{align*}
In what follows we introduce the notation
\begin{equation}\label{xi:uv}
\xi(u, v) := \lim_{\varepsilon \to 0+} \left( \int_{|x| > \varepsilon} \nabla u \nabla v \, dx - \frac{(N-2)^2}{4} \int_{|x| > \varepsilon} \frac{u v}{|x|^2} \, dx - \frac{N-2}{2} \varepsilon^{-1} \int_{|x|=\varepsilon} u v \, dS \right)
\end{equation}
and
$$
\xi(u) := \xi(u,u) = \lim_{\varepsilon \to 0+} \left( \int_{|x| > \varepsilon} |\nabla u|^2 \, dx - \frac{(N-2)^2}{4} \int_{|x| > \varepsilon} \frac{u^2}{|x|^2} \, dx - \frac{N-2}{2} \varepsilon^{-1} \int_{|x|=\varepsilon} u^2 \, dS \right).
$$
Such a characterization plays a role for functions $u \in X^1 (\R^N) \setminus H^1 (\R^N)$. Otherwise, namely if $u \in H^1 (\R^N)$, it is more convenient to use the formula \eqref{normX1-oryg} instead of \eqref{normX1}.

\begin{Lem}\label{L:groundstate}
Let $u \in X^1(\R^N)$ and $v := |x|^{\frac{N-2}2} u$. Then $u, v \in H^1_{\loc}(\R^N \setminus \{0\})$ and
\begin{equation}\label{eq:groundstate}
\xi(u) = \int_{\R^N} |x|^{-(N-2)} |\nabla v|^2 \, dx .
\end{equation}
In particular, $\xi(u) > 0$ for every $u \in X^1(\R^N) \setminus \{0\}$.
\end{Lem}

\begin{proof}
In what follows we set $a := \frac{N-2}{2}$, so that $a^2 = \frac{(N-2)^2}{4}$ and $2a + 2 = N$. 
Let first $u \in \cC_0^\infty(\R^N)$ and $v = |x|^a u$. A direct computation gives, for $x \neq 0$,
$$
|\nabla u|^2 - a^2 \frac{u^2}{|x|^2} = |x|^{-2a}|\nabla v|^2 - a |x|^{-N} x \cdot \nabla (v^2),
$$
and $\div \left( |x|^{-N} x \right) = 0$ in $\R^N \setminus \{0\}$. Hence, integrating by parts on $\{|x| > \varepsilon\}$,
$$
\int_{|x|>\varepsilon} |\nabla u|^2 - a^2 \frac{u^2}{|x|^2} \, dx = \int_{|x|>\varepsilon} |x|^{-2a}|\nabla v|^2 \, dx + \frac{a}{\varepsilon} \int_{|x|=\varepsilon} u^2 \, dS .
$$
Since $u$ is bounded, the last term is $O(\varepsilon^{N-2})$, and letting $\varepsilon \to 0^+$ we obtain \eqref{eq:groundstate} for $u \in \cC_0^\infty(\R^N)$.

Now let $u \in X^1(\R^N)$ and take $u_n \in \cC_0^\infty(\R^N)$ with $u_n \to u$ in $X^1(\R^N)$. Put $v_n := |x|^a u_n$. By \eqref{eq:groundstate} applied to $u_n - u_m$, $(\nabla v_n)$ is a Cauchy sequence in $L^2(\R^N, |x|^{-2a}dx)$, so $\nabla v_n \to w$ in this space for some $w$. On the other hand, $u_n \to u$ in $L^2(\R^N)$, so $v_n \to v$ in $L^2_{\loc}(\R^N\setminus\{0\})$. Hence $w = \nabla v$ in $\cD'(\R^N\setminus\{0\})$ and $v \in H^1_{\loc}(\R^N\setminus\{0\})$, since $|x|^{-2a}$ is bounded away from zero on compact subsets of $\R^N \setminus \{0\}$. Consequently $u = |x|^{-a} v \in H^1_{\loc}(\R^N\setminus\{0\})$ and
$$
\xi(u) = \lim_{n\to\infty} \xi(u_n) = \lim_{n\to\infty} \int_{\R^N} |x|^{-2a}|\nabla v_n|^2 \, dx = \int_{\R^N} |x|^{-2a}|\nabla v|^2 \, dx.
$$
Finally, if $\xi(u) = 0$, then $\nabla v = 0$ in the connected set $\R^N \setminus \{0\}$, so $u = C|x|^{-a}$, and $u \in L^2(\R^N)$ forces $C = 0$.
\end{proof}

Observe that $X^1 (\R^N)$ is not translation invariant, and that the maximal translation invariant subspace is $H^1 (\R^N)$. Indeed, if $u \in H^1 (\R^N) \subset X^1 (\R^N)$, then for every $y \in \R^N$, $u(\cdot - y) \in H^1 (\R^N) \subset X^1 (\R^N)$. Moreover we have the following fact.

\begin{Lem}
If $u \in X^1 (\R^N) \setminus H^1 (\R^N)$, then for every $y \in \R^N \setminus \{0\}$, $u(\cdot -y) \not\in X^1 (\R^N)$.
\end{Lem}

\begin{proof}
Let $u \in X^1 (\R^N) \setminus H^1 (\R^N)$. Note that, thanks to Lemma \ref{L:groundstate} (see also the characterization of the norm \eqref{normX1} from \cite{TrachanasZographopoulos}), we know that $\nabla u \in L^2 (\R^N \setminus B(0,\varepsilon))$ for every $\varepsilon > 0$. Moreover, since $u \not\in H^1 (\R^N)$, $\nabla u \not\in L^2 (B(0,\varepsilon))$ for every $\varepsilon > 0$. Fix $y \in \R^N \setminus \{0\}$ and choose $\varepsilon < \frac{|y|}{2}$. Then $\nabla u(\cdot - y) \not\in L^2 (B(y,\varepsilon))$, but $B(y,\varepsilon) \subset \R^N \setminus B(0,\varepsilon)$, so $\nabla u (\cdot -y ) \not \in L^2 (\R^N \setminus B(0,\varepsilon))$, meaning that $u (\cdot -y ) \not\in X^1 (\R^N)$.
\end{proof}

This fact makes the argument based on the Lions' concentration-compactness lemma and translations significantly more challenging than in the standard case. We overcome this difficulty by the use of continuous embeddings into $H^s (\R^N)$, which is translation-invariant, and by making use of the regularity theory to find that limits of translated Poho\v{z}aev-Palais-Smale sequences are indeed in $X^1 (\R^N)$, and even in $H^1_{\mathrm{loc}} (\R^N)$. For this purpose we show that such a limit is a solution to an elliptic problem without the singular potential in the distributional sense. Then we make use of Weyl's lemma for the Poisson problem to see that such a solution lies in $H^1_{\mathrm{loc}} (\R^N)$. Then, to show even more - the limit point belongs globally to $H^1(\R^N)$ - we make use of a variant of the Caccioppoli inequality.

\section{Variational setting}\label{sect:3}

Define $J : X^1 (\R^N) \rightarrow \R$ by
$$
J(u) = \frac12 \xi(u) - \int_{\R^N} G(u) \, dx.
$$
Then, in view of (G1)--(G3) and embeddings \eqref{eq:embeddings}, $J$ is of $\cC^1$ class and
$$
J'(u)(v) = \xi(u,v) - \int_{\R^N} g(u) v \, dx.
$$
Observe that on $X^1 (\R^N)$ we can introduce an equivalent norm
$$
\|u\|_\ell^2 := \xi(u) + \ell \int_{\R^N} u^2 \, dx
$$
and the scalar product
$$
\langle u, v \rangle_\ell := \xi(u,v) + \ell \int_{\R^N} uv \, dx.
$$
Then $J$ and its derivative can be rewritten as
$$
J(u) = \frac12 \|u\|_\ell^2 - \int_{\R^N} F(u) \, dx, \quad J'(u)(v) = \langle u,v\rangle_\ell - \int_{\R^N} f(u) v \, dx, \quad u,v \in X^1 (\R^N).
$$
Note that if $u \in X^1 (\R^N)$ is a critical point of $J$, then $u$ is a weak solution to \eqref{eq:main} on $\R^N \setminus \{0\}$. Indeed, take any test function $\varphi \in \cC_0^\infty (\R^N \setminus \{0\})$ and consider the equality $J'(u)(\varphi) = 0$. Then
$$
\xi(u, \varphi) = \int_{\R^N} g(u) \varphi \, dx.
$$
Recall \eqref{xi:uv} and, since $\supp \varphi \subset \R^N \setminus B(0,\varepsilon)$ for sufficiently small $\varepsilon > 0$, we see that
$$
\int_{|x| > \varepsilon} \nabla u \nabla \varphi \, dx = \int_{\R^N} \nabla u \nabla \varphi \, dx, \quad \int_{|x| > \varepsilon} \frac{u \varphi}{|x|^2} \, dx = \int_{\R^N} \frac{u \varphi}{|x|^2} \, dx, \quad \int_{|x|=\varepsilon} u \varphi \, dS = 0.
$$
Hence
$$
\int_{\R^N} \nabla u \nabla \varphi \, dx - \frac{(N-2)^2}{4} \int_{\R^N} \frac{u \varphi}{|x|^2} \, dx = \int_{\R^N} g(u) \varphi \, dx
$$
and $u$ is a weak solution to \eqref{eq:main} on $\R^N \setminus \{0\}$.

We are going to show that $J$ has the mountain pass geometry.

\begin{Lem}\label{r}
There is $r > 0$ such that $\inf_{\|u\|_\ell =r} J(u) > 0$. 
\end{Lem}

\begin{proof}
Note that, from (G1)--(G3) there is $C > 0$ such that $G(u) \leq - \frac{\ell}{4} |u|^2 + C |u|^{p}$. Hence
\begin{align*}
J(u) &\geq \frac12 \xi(u) + \frac{\ell}{4} \int_{\R^N} |u|^2 \, dx - C \int_{\R^N} |u|^{p} \, dx \\
&= \frac 14 \xi(u)+\frac{1}{4} \|u\|_\ell^2 - C \int_{\R^N} |u|^{p} \, dx \geq \frac{1}{4} \|u\|_\ell^2 - C \|u\|_\ell^p.
\end{align*}
Since $p > 2$, we get $\inf_{\|u\|_\ell =r} J(u) > 0$ for sufficiently small $r > 0$.
\end{proof}

\begin{Lem}\label{v}
There is $v \in X^1 (\R^N)$ such that $J(v) < 0$ and $\|v\|_\ell > r$, where $r$ is given in Lemma \ref{r}.
\end{Lem}

\begin{proof}
Following the idea from \cite{BL}, for any $R > 0$, we define
$$
w_R(x) := \left\{ \begin{array}{ll}
\zeta_0 & \quad |x| < R \\
\zeta_0 (R+1-|x|) & \quad R \leq |x| \leq R+1 \\
0 & \quad |x| > R+1
\end{array} \right.,
$$
where $\zeta_0$ is defined in (G4).
It is clear that $w_R \in H^1 (\R^N) \subset X^1 (\R^N)$ for every $R > 0$. Moreover $w_R(x) \geq 0$ for $x \in \R^N$. Recall that
$$
| B(0, R) | = c_N R^N
$$
for some $c_N > 0$. Hence $| B(0,R+1) \setminus B(0,R) | = c_N \left( (R+1)^N - R^N \right) \lesssim  R^{N-1}$. Hence
\begin{align*}
\int_{\R^N} G(w_R) \, dx &\geq G(\zeta_0) c_N R^N - \sup_{\zeta \in [0, \zeta_0]} |G(\zeta)| \alpha_N R^{N-1}, \quad \mbox{for some } \alpha_N > 0, \\
\int_{\R^N} |\nabla w_R|^2 \, dx & \lesssim R^{N-1}, \\
\int_{\R^N} |w_R|^2 \, dx & \gtrsim R^N.
\end{align*}
Then
$$
J(w_R) \leq \frac12 \int_{\R^N} |\nabla w_R|^2 \, dx - \int_{\R^N} G(w_R) \, dx \leq C_1 R^{N-1} - C_2 R^{N} + C_3 R^{N-1}
$$
with $C_2 > 0$. Moreover 
$$
\| w_R \|_\ell^2 \geq \ell \int_{\R^N} |w_R|^2 \, dx \gtrsim R^N.
$$
Thus $J(w_R) < 0$ and $\|w_R\|_\ell > r$ for $R > 0$ large enough.
\end{proof}

Lemma \ref{r} and Lemma \ref{v} imply that $J$ has the mountain pass geometry. We define the mountain pass level
\begin{equation}\label{def:c}
c := \inf_{\gamma \in \Gamma} \sup_{t \in [0,1]} J(\gamma(t)),
\end{equation}
where the family of paths $\Gamma$ is given by
\begin{equation}\label{def:Gamma}
\Gamma := \left\{ \gamma \in \cC ([0,1]; X^1 (\R^N)) \ : \ \gamma(0) = 0, \ J(\gamma(1)) < 0 \right\}.
\end{equation}

Note that
\begin{equation}\label{eq:c-positive}
c \geq \inf_{\|u\|_\ell = r} J(u) > 0 .
\end{equation}
Indeed, by the proof of Lemma \ref{r} we may assume that $J(u) > 0$ for $0 < \|u\|_\ell \leq r$. Hence every $\gamma \in \Gamma$ satisfies $\|\gamma(1)\|_\ell > r$ and, by continuity, $\|\gamma(t)\|_\ell = r$ for some $t \in (0,1)$.

The following fact is a well-known consequence of the deformation lemma and we include it here for the reader's convenience.

\begin{Lem}\label{lemma:critical_point_on_path}
If there is $\gamma \in \Gamma$ and $t^*  \in (0,1)$ with 
$$
c = J(\gamma(t^*)) > J(\gamma(t)) \quad \mbox{for } t \neq t^*,
$$
then $J'(\gamma(t^*)) = 0$.
\end{Lem}

\begin{proof}
Assume by contradiction that $J'(\gamma(t^*)) \neq 0$. Since $J$ is of $\cC^1$ class, we may choose $\delta > 0$ and $\varepsilon\in(0,\frac c2)$ such that
$$
\inf \{ \| J'(v) \| \ : \ \|v - \gamma(t^*)\| \leq 2\delta \} > \frac{8\varepsilon}{\delta}.
$$
Then, by the deformation lemma \cite[Lemma 2.3]{W}, we may find a flow $\eta \in \cC ([0,1] \times X^1 (\R^N); X^1 (\R^N))$ such that $\eta(1, \gamma(\cdot)) \in \Gamma$ and $J(\eta(1,\gamma(t^*))) \leq c - \varepsilon < c$. Moreover, for every $t \neq t^*$, $J(\eta(1,\gamma(t))) \leq J(\gamma(t)) < c$. Hence, for every $t \in [0,1]$ we have $J(\eta(1,\gamma(t))) < c$. Thus $\max_{t \in [0,1]} J(\eta(1,\gamma(t))) < c$ and we obtain a contradiction.
\end{proof}
We define
\[
    P(u) 
    = \xi(u) - 2^* \int_{\R^N} G(u) \, dx.
    \]

The following result shows that, in spite of the presence of the singular term in $\xi$, every critical point of $J$ satisfies the Poho\v{z}aev identity, i.e. $P(u)=0$.

\begin{Prop}\label{P:Pohozaev}
Suppose that (G1)--(G3) hold. If $u \in X^1(\R^N)$ satisfies $J'(u) = 0$, then $P(u) = 0$. In particular, every nontrivial critical point $u$ of $J$ belongs to $\cM$ and $J(u) = \frac1N \xi(u) > 0$.
\end{Prop}

\begin{proof}
By the proof of Theorem \ref{Th:Main1-Existence}(c), which uses only the equation $J'(u)=0$, see Section \ref{sect:6}, we have $u \in W^{2,q}_{\loc}(\R^N\setminus\{0\}) \cap \cC^{1}(\R^N\setminus\{0\})$ for all $q<\infty$. Let $v := |x|^{\frac{N-2}{2}} u$ and $\rho(x) := |x|^{2-N}$. A direct computation shows that
$$
\Delta u + \frac{(N-2)^2}{4} \frac{u}{|x|^2} = |x|^{\frac{N-2}{2}} \div (\rho \nabla v) \quad \mbox{in } \R^N \setminus \{0\},
$$
so that $v \in W^{2,2}_{\loc}(\R^N\setminus\{0\}) \cap \cC^1(\R^N \setminus\{0\})$ satisfies
\begin{equation}\label{eq:v-equation}
-\div(\rho\nabla v) = |x|^{-\frac{N-2}{2}} g(u) \quad \mbox{a.e. in } \R^N \setminus \{0\}.
\end{equation}
Fix $0 < \varepsilon < R$, let $A := \{ \varepsilon < |x| < R\}$, and let $\nu$ denote the outer unit normal to $\partial A$. We multiply \eqref{eq:v-equation} by $x\cdot\nabla v$ and integrate over $A$. Since $\nabla v \cdot \nabla(x\cdot \nabla v) = |\nabla v|^2 + \frac12 x \cdot \nabla |\nabla v|^2$ and $\div(\rho x) = 2\rho$, we have
$$
\int_A \rho \nabla v \cdot \nabla (x\cdot\nabla v) \, dx = \frac12 \int_{\partial A} \rho\, (x\cdot\nu) |\nabla v|^2 \, dS,
$$
hence
$$
\int_A -\div(\rho\nabla v)\, (x\cdot\nabla v) \, dx = \int_{\partial A} \rho \left( \frac12 (x\cdot\nu)|\nabla v|^2 - (\nabla v \cdot \nu)(x \cdot \nabla v) \right) dS =: B(\varepsilon,R).
$$
On the other hand, since $\nabla u = |x|^{-\frac{N-2}{2}}\nabla v - \frac{N-2}{2}|x|^{-\frac{N-2}{2}-2} x\, v$, we have $|x|^{-\frac{N-2}{2}} g(u)\, (x\cdot\nabla v) = x \cdot \nabla \big(G(u)\big) + \frac{N-2}{2} g(u) u$, so that
$$
\int_A |x|^{-\frac{N-2}2} g(u)\, (x\cdot\nabla v) \, dx = -N \int_A G(u) \, dx + \frac{N-2}{2} \int_A g(u) u \, dx + \int_{\partial A} (x \cdot \nu) G(u) \, dS.
$$
Consequently
\begin{equation}\label{eq:Poh-annulus}
B(\varepsilon,R) - \int_{\partial A} (x \cdot \nu) G(u) \, dS = -N \int_A G(u) \, dx + \frac{N-2}{2} \int_A g(u) u \, dx.
\end{equation}
Put $\phi(r) := \int_{|x| = r} \left( \rho|\nabla v|^2 + |G(u)| \right) dS$ for $r > 0$. The left-hand side of \eqref{eq:Poh-annulus} is bounded in absolute value by $\frac32 \left( \varepsilon \phi(\varepsilon) + R\phi(R) \right)$. By Lemma \ref{L:groundstate} and since $G(u) \in L^1(\R^N)$,
$$
\int_0^\infty \phi(r) \, dr = \int_{\R^N} \rho |\nabla v|^2 + |G(u)| \, dx = \xi(u) + \int_{\R^N}|G(u)|\,dx < \infty,
$$
hence there are sequences $\varepsilon_k \to 0^+$ and $R_k \to \infty$ with $\varepsilon_k \phi(\varepsilon_k) \to 0$ and $R_k\phi(R_k) \to 0$ (otherwise $\phi(r) \geq c_0/r$ for some $c_0>0$ near $0$ or near $\infty$, which is not integrable). Since $G(u), g(u)u \in L^1(\R^N)$, passing to the limit in \eqref{eq:Poh-annulus} along these sequences we obtain
$$
N \int_{\R^N} G(u) \, dx = \frac{N-2}{2} \int_{\R^N} g(u) u \, dx = \frac{N-2}{2}\, \xi(u),
$$
where in the last equality we used $J'(u)(u) = 0$. This is exactly $P(u) = 0$. If $u \neq 0$, then $\xi(u) > 0$ by Lemma \ref{L:groundstate}, so $u \in \cM$ and $J(u) = \frac12 \xi(u) - \frac{1}{2^*}\xi(u) = \frac1N \xi(u) > 0$.
\end{proof}

To find a Poho\v{z}aev-Palais-Smale sequence, we will use the technique of the augmented functional introduced by Jeanjean \cite{Jeanjean} (see also \cite{HirataIkomaTanaka, JeanjeanTanaka}). Then, thanks to the density of $H^1 (\R^N)$ in $X^1(\R^N)$, we can choose such a sequence in $H^1(\R^N)$, which will later simplify notations and computations, but it is not necessary for the argument to work.

\begin{Lem}\label{Lem:ExistencePPS}
There is $(u_n) \subset H^1 (\R^N)$ such that
$$
J(u_n) \to c, \quad J'(u_n) \to 0, \quad P(u_n) \to 0.
$$
\end{Lem}

\begin{proof}
Set $\Phi : \R \times X^1 (\R^N) \rightarrow X^1 (\R^N)$ given by
$$
\Phi(t, u)(x) := u(e^{-t} x).
$$
Then
$$
J(\Phi(t, u)) = \frac{e^{(N-2)t}}{2} \xi(u) - e^{Nt} \int_{\R^N} G(u) \, dx.
$$
It is clear that $J \circ \Phi$ is of class $\cC^1 (\R \times X^{1}(\R^N))$. 
Then we consider the following family of paths
$$
\widetilde{\Ga} := \left\{ \widetilde{\gamma} \in \cC ([0,1]; \R \times X^1 (\R^N)) \ : \ \widetilde{\gamma}(0)=(0,0),\, (J \circ \Phi)(\widetilde{\gamma}(1)) < 0\right\}.
$$
Note that $\Ga = \{ \Phi \circ \widetilde{\gamma} \ : \ \widetilde{\gamma} \in \widetilde{\Ga} \}$. Hence
$$
c = \inf_{\widetilde{\gamma} \in \widetilde{\Ga}} \sup_{t \in [0,1]} (J \circ \Phi)(\widetilde{\gamma}(t)).
$$
For every $\varepsilon \in \left(0, \frac{c}{2} \right)$ we may choose a path $\gamma \in \Ga$ such that $\sup_{t \in [0,1]} (J \circ \Phi)(0,\gamma(t)) \leq c + \varepsilon$. Hence, by Ekeland's variational principle \cite[Theorem 2.8]{W}, there exists $(t, u) \in \R \times X^1 (\R^N)$ such that
\begin{align*}
c - 2\varepsilon \leq (J \circ \Phi)(t,u) \leq c + 2\varepsilon, \\
\inf_{s \in [0,1]} \left( |t|^2 + \| u - \gamma(s) \|^2 \right)^{1/2} \leq 2 \sqrt{\varepsilon}, \\
\left\| (J \circ \Phi)' (t,u) \right\| \leq 2 \sqrt{\varepsilon}
\end{align*}
and we can construct a sequence $(t_n, v_n) \subset \R \times X^1 (\R^N)$ such that
$$
t_n \to 0, \quad (J \circ \Phi)(t_n, v_n) \to c, \quad (J \circ \Phi)'(t_n, v_n) \to 0.
$$
Set $\widetilde{u}_n := \Phi(t_n, v_n)$. Then, for every $(h, w) \in \R \times X^1 (\R^N)$,
$$
(J \circ \Phi)'(t_n, v_n) (h,w) = J'(\widetilde{u}_n)(\Phi(t_n,w)) + \frac{N-2}{2} P(\widetilde{u}_n)h. 
$$
It is clear that $J(\widetilde{u}_n) = J(\Phi(t_n, v_n)) \to c$. Taking $w = 0$ and $h=1$ we obtain that
$$
o(1) = (J \circ \Phi)'(t_n, v_n) (1, 0) = \frac{N-2}{2} P(\widetilde{u}_n)
$$
so that $P(\widetilde{u}_n) \to 0$. It is also classical to check that $J'(\widetilde{u}_n) \to 0$. Since $H^1(\R^N)$ is dense in $X^1(\R^N)$ and $J$, $J'$, $P$ are continuous at $\widetilde{u}_n$, for each $n$ we can find $u_n \in H^1(\R^N)$ such that
\[
|J(u_n) - J(\widetilde{u}_n)| + \|J'(u_n) - J'(\widetilde{u}_n)\| + |P(u_n) - P(\widetilde{u}_n)| \leq \frac1n .
\]
Then $J(u_n) \to c$, $J'(u_n) \to 0$ and $P(u_n) \to 0$.

\end{proof}

\section{Profile decomposition of Poho\v{z}aev-Palais-Smale sequences}\label{sect:4}

In this section, we study the profile decomposition of bounded Poho\v{z}aev-Palais-Smale sequences. We begin with several lemmas that will be useful in proving the forthcoming decomposition result.

\begin{Lem}\label{L:scalar_product_translation}
If $u \in X^1 (\R^N)$, $w \in H^1 (\R^N)$ and $(y_n) \subset \R^N$ satisfy $|y_n| \to \infty$, then
\[
\lim_{n\to\infty} \langle u,w(\cdot-y_n)\rangle_\ell = 0.
\]
\end{Lem}

\begin{proof}
Recall that $w(\cdot -y_n) \weakto 0$ in $H^1 (\R^N)$. Indeed, fix $\varphi \in \cC_0^\infty (\R^N)$ and note that
\begin{align*}
&\quad \left| \int_{\R^N} \nabla w (\cdot - y_n) \nabla \varphi + w(\cdot - y_n) \varphi \, dx \right| = \left| \int_{-y_n + \supp \varphi} \nabla w \nabla \varphi ( \cdot + y_n ) + w \varphi (\cdot + y_n) \, dx \right| \\
&\leq \underbrace{\left( \int_{-y_n + \supp \varphi} |\nabla w|^2 + w^2 \, dx \right)^{1/2}}_{\to 0} \left( \int_{\R^N} |\nabla \varphi|^2 + \varphi^2 \, dx \right)^{1/2} \to 0.
\end{align*}
Now, take any $v \in H^1 (\R^N)$ and let $\varphi_k \in \cC_0^\infty (\R^N)$ be the sequence such that $\varphi_k \to v$ in $H^1(\R^N)$. Then
\begin{align*}
&\quad \left| \int_{\R^N} \nabla w (\cdot - y_n) \nabla v + w(\cdot - y_n) v \, dx \right| \\
&\leq \left| \int_{\R^N} \nabla w (\cdot - y_n) \nabla (v-\varphi_k) + w(\cdot - y_n) (v-\varphi_k) \, dx \right|  + \left| \int_{\R^N} \nabla w (\cdot - y_n) \nabla \varphi_k + w(\cdot - y_n) \varphi_k \, dx \right| \\
&\leq \|w\|_{H^1(\R^N)} \|v-\varphi_k\|_{H^1(\R^N)} + o_n(1).
\end{align*}
Now, taking $k\to\infty$ we obtain 
$$
\int_{\R^N} \nabla w (\cdot - y_n) \nabla v + w(\cdot - y_n) v \, dx \to 0.
$$
Thus, indeed $w(\cdot -y_n) \weakto 0$ in $H^1 (\R^N)$. Then, since the embedding $H^1 (\R^N) \subset X^1 (\R^N)$ is continuous, thanks to \cite[Theorem 3.10]{Brezis}, we get $w(\cdot - y_n) \weakto 0$ in $X^1 (\R^N)$. Hence $\langle u, w(\cdot - y_n) \rangle_\ell \to 0$ and the proof is completed.
\end{proof}

We will apply the following lemma several times.

\begin{Lem}[{\cite[Lemma 3.1]{GuoMederski}}] \label{L:MederskiGuo}
If $u \in H^1 (\R^N)$ and $(y_n) \subset \R^N$ satisfy $|y_n|\to\infty$, then
$$
\lim_{n\to\infty} \int_{\R^N} \frac{|u(\cdot - y_n)|^2}{|x|^2} \, dx = 0.
$$
\end{Lem}

\begin{Lem} \label{H1-translations-bounded}
If $u \in H^1 (\R^N)$ and $(y_n) \subset \R^N$ satisfy $|y_n| \to \infty$, then
$$
\lim_{n\to\infty} \| u (\cdot - y_n) \| = \| u \|_{H^1 (\R^N)}.
$$
Moreover, as a simple consequence,
$$
\lim_{n\to\infty} \|u (\cdot - y_n)\|_\ell^2 = \| \nabla u\|_{L^2 (\R^N)}^2 + \ell \| u \|_{L^2 (\R^N)}^2.
$$
\end{Lem}

\begin{proof}
Observe that 
\begin{align*}
\| u (\cdot - y_n) \|^2 &= \int_{\R^N} |\nabla u(\cdot - y_n)|^2 \, dx - \frac{(N-2)^2}{4} \int_{\R^N} \frac{|u(\cdot - y_n)|^2}{|x|^2} \, dx + \int_{\R^N} |u(\cdot - y_n)|^2 \, dx \\
&= \|u\|_{H^1 (\R^N)}^2 - \frac{(N-2)^2}{4} \int_{\R^N} \frac{|u(\cdot - y_n)|^2}{|x|^2} \, dx. 
\end{align*}
Moreover, from Lemma \ref{L:MederskiGuo},
$$
\int_{\R^N} \frac{|u(\cdot - y_n)|^2}{|x|^2} \, dx \to 0
$$
and the proof is completed.
\end{proof}

From Lions’ concentration-compactness principle, we easily obtain the following corollary, which can be viewed as a variant of the concentration–compactness principle in $X^1 (\R^N)$, cf. \cite{Lions1}.

\begin{Lem}\label{lionsLem}
Let $(u_n) \subset X^1 (\R^N)$ be a bounded sequence with
$$
\lim_{n \to +\infty} \sup_{z \in \R^N} \int_{B(z, r)} |u_n|^2 \, dx = 0
$$
for some $r > 0$. Then $u_n \to 0$ in $L^t (\R^N)$ for $t \in (2, 2^*)$.
\end{Lem}

\begin{proof}
Fix $t \in (2, 2^*)$. Then there is $\varepsilon \in (0,1)$ such that $t < \frac{2N}{N-2(1-\varepsilon)}$. Since $X^1 (\R^N) \subset H^{1-\varepsilon} (\R^N)$, we know that $u_n \to 0$ in $L^s (\R^N)$ for $2 < s < \frac{2N}{N-2(1-\varepsilon)}$. Therefore, $u_n \to 0$ in $L^t (\R^N)$.
\end{proof}

In the profile decomposition presented below, we construct weak solutions to the \textit{limiting problem}
$$
-\Delta u = g(u).
$$
Since we work in $X^1 (\R^N)$, while the natural energy space for this equation is $H^1 (\R^N)$, we first need a few regularity results.

We recall the following fact, which is a consequence of Weyl's lemma.

\begin{Th}[{\cite[Lemma 3]{DiFratta}}]\label{T:dystrybucje_na_H1}
Let $\Omega \subset \R^N$ be an open set and $h \in H^{-1} (\Omega)$. Suppose that $u \in L^2 (\Omega)$ satisfies
$$
-\Delta u = h \quad \mbox{in } \cD'(\Omega).
$$
Then $u \in H^{1}_{\mathrm{loc}} (\Omega)$.
\end{Th}

Applying this theorem for $\Omega=\R^N$, we can get the following result in the setting that we are interested in.

\begin{Th}\label{th:usuwanie_osobliwosci}
Suppose that $g$ satisfies (G1)--(G3). Suppose that $u \in H^s (\R^N)$, where $s \in (0,1)$ is chosen so that $p < \frac{2N}{N-2s}$, satisfies
$$
-\Delta u = g(u) \quad \mbox{in } \cD'(\R^N ).
$$
Then $u \in H^1 (\R^N)$ and $u$ is a weak solution to $-\Delta u = g(u)$ on $\R^N$.
\end{Th}

\begin{proof}
\textbf{Step 1.} \textit{The distributional solution $u$ belongs to $H^1_{\loc} (\R^N)$}. 

Let $u$ be as in the statement. Since
\[
|g(u)| \lesssim |u|+|u|^{p-1} \in L^2(\R^N) + L^\frac{p}{p-1} (\R^N) \subset H^{-1} (\R^N),
\]
by Theorem \ref{T:dystrybucje_na_H1} we know that $u \in H^1_{\loc} (\R^N)$.

\textbf{Step 2.} \textit{The solution $u$ lies globally in $H^1 (\R^N)$.}

Now, we will show that $u \in H^1 (\R^N)$. First, we will show the following variant of the Caccioppoli inequality (cf. \cite[Chapter 4]{Giaquinta}, \cite{Iwaniec})
$$
\int_{B(0,R)} |\nabla u|^2 \, dx \lesssim \int_{B(0,2R)} |g(u) u| \, dx + \frac{1}{R^2} \int_{B(0,2R) \setminus B(0,R)} u^2 \, dx.
$$
Take $R > 0$ and $\eta_R \in \cC_0^\infty (\R^N)$ such that
$$
\eta_R \equiv 1 \mbox{ on } B(0,R), \quad \eta_R \equiv 0 \mbox{ outside of } B(0,2R)
$$
and
$$
0 \leq \eta_R \leq 1, \quad |\nabla \eta_R | \leq \frac{c}{R}
$$
for some $c \geq 1$. Then $\eta_R^2 u \in H^1_0 (B(0,2R))$. Take a sequence $\varphi_n \in \cC_0^\infty (B(0,2R))$ such that $\varphi_n \to \eta_R^2 u$ in $H^1_0(B(0,2R))$. Testing the equation with $\varphi_n$ we get
$$
\int_{\R^N} \nabla u \nabla \varphi_n \, dx = \int_{\R^N} g(u) \varphi_n \, dx.
$$
Since $\varphi_n \to \eta_R^2 u$ in the strong topology of $H^1_0(B(0,2R))$, passing to the limit, we easily get that
$$
\int_{\R^N} \nabla u \nabla (\eta_R^2 u) \, dx = \int_{\R^N} g(u) \eta_R^2 u \, dx.
$$
It can be rewritten as
$$
\int_{\R^N} \eta_R^2 |\nabla u|^2 \, dx + 2 \int_{\R^N} \eta_R u \nabla u \cdot \nabla \eta_R \, dx = \int_{\R^N} g(u) \eta_R^2 u \, dx.
$$
Now, using Young's inequality we obtain
\begin{align*}
\left| 2 \int_{\R^N} \eta_R u \nabla u \cdot \nabla \eta_R \, dx \right| &\leq  \int_{\R^N} (\eta_R |\nabla u|) (2| \nabla \eta_R| |u|) \, dx  \leq \frac{1}{2} \int_{\R^N} \eta_R^2 |\nabla u|^2 \, dx + 2 \int_{\R^N} |\nabla \eta_R|^2 |u|^2 \, dx \\
&\leq \frac12 \int_{\R^N} \eta_R^2 |\nabla u|^2 \, dx + \frac{2c^2}{R^2} \int_{B(0,2R) \setminus B(0,R)} u^2 \, dx.
\end{align*}
Hence
\begin{align*}
\int_{\R^N} \eta_R^2 |\nabla u|^2 \, dx \leq \int_{\R^N} |g(u) u| \eta_R^2 \, dx + \frac12 \int_{\R^N} \eta_R^2 |\nabla u|^2 \, dx + \frac{2c^2}{R^2} \int_{B(0,2R) \setminus B(0,R)} u^2 \, dx
\end{align*}
or equivalently
\begin{align*}
\frac12 \int_{\R^N} \eta_R^2 |\nabla u|^2 \, dx \leq \int_{B(0,2R)} |g(u) u| \eta_R^2 \, dx + \frac{2c^2}{R^2} \int_{B(0,2R) \setminus B(0,R)} u^2 \, dx.
\end{align*}
Thus
\begin{align*}
\int_{B(0,R)} |\nabla u|^2 \, dx &\leq \int_{\R^N} \eta_R^2 |\nabla u|^2 \, dx \lesssim \int_{B(0,2R)} |g(u) u| \eta_R^2 \, dx + \frac{1}{R^2} \int_{B(0,2R) \setminus B(0,R)} u^2 \, dx \\
&\leq \int_{B(0,2R)} |g(u) u| \, dx + \frac{1}{R^2} \int_{B(0,2R) \setminus B(0,R)} u^2 \, dx.
\end{align*}
Hence, we obtained the following variant of the Caccioppoli inequality 
\begin{equation}\label{caccioppoli}
    \int_{B(0,R)} |\nabla u|^2 \, dx \lesssim \int_{B(0,2R)} |g(u) u| \, dx + \frac{1}{R^2} \int_{B(0,2R) \setminus B(0,R)} u^2 \, dx.
\end{equation}
Recall that
$$
|g(u)| \lesssim |u| + |u|^{p-1}.
$$
Thus, from the Caccioppoli inequality \eqref{caccioppoli} we obtain
\begin{align*}
\int_{B(0,R)} |\nabla u|^2 \, dx &\lesssim \int_{B(0,2R)} |g(u) u| \, dx + \frac{1}{R^2} \int_{B(0,2R) \setminus B(0,R)} u^2 \, dx \\
&\lesssim \|u\|_{L^2 (\R^N)}^2 + \|u\|_{L^p (\R^N)}^p  + \frac{1}{R^2} \|u\|_{L^2 (\R^N)}^2
\end{align*}
Now, taking $R \to \infty$, from the monotone convergence theorem
$$
\int_{\R^N} |\nabla u|^2 \, dx \lesssim \|u\|_{L^2 (\R^N)}^2 + \|u\|_{L^p (\R^N)}^p < \infty
$$
and therefore $u \in H^1 (\R^N)$.
\end{proof}

We introduce the energy and Poho\v{z}aev functionals for the limiting equation. We define ${J_\infty : H^1 (\R^N) \rightarrow \R}$ and $P_\infty : H^1 (\R^N) \rightarrow \R$ as
\begin{align*}
J_\infty(u) &:= \frac12 \int_{\R^N} |\nabla u|^2 \, dx + \frac{\ell}{2} \int_{\R^N} u^2 \, dx - \int_{\R^N} F(u) \, dx, \\
P_{\infty}(u) &:= \int_{\R^N} |\nabla u|^2 \, dx + \frac{2^* \ell}{2} \int_{\R^N} u^2 \, dx - 2^* \int_{\R^N}  F(u) \, dx.
\end{align*}

\begin{Rem}\label{R:problem_graniczny_wlasnosci}
It is well known that every critical point $w \in H^1 (\R^N)$ of $J_\infty$ satisfies $P_\infty(w) = 0$.
\end{Rem}

\begin{Th}\label{Th:splitting}
Suppose that $(u_n) \subset H^1 (\R^N)$ is bounded in $X^1(\R^N)$, Poho\v{z}aev-Palais-Smale sequence for $J$ at level $c > 0$. Then, passing to a subsequence, there are an integer $k \geq 0$, $u_0 \in X^1 (\R^N)$, and sequences $(y_n^j) \subset \Z^N$, $w^j \in H^1 (\R^N)$ for $j\in \{1,\ldots, k\}$ such that

\begin{itemize}
\item[(a)] $u_n \weakto u_0$, $J'(u_0)=0$ and $P(u_0)=0$;
\item[(b)] $|y_n^j| \to \infty$ and $|y_n^j - y_n^{j'}|\to \infty$ for $j \neq j'$;
\item[(c)] $w^j \neq 0$, $w^j$ is a weak solution to
$$
-\Delta w^j = g(w^j) \quad \mbox{in } \R^N,
$$
and is a critical point of $J_{\infty}$ for each $1 \leq j \leq k$;
\item[(d)] $u_n - u_0 - \sum_{j=1}^k w^j (\cdot - y_n^j) \to 0$ in $X^1 (\R^N)$;
\item[(e)] $\|u_n\|^2_\ell\to \|u_0\|^2_\ell+\sum_{j=1}^k  \left( \|\nabla w^j\|_{L^2 (\R^N)}^2 + \ell \| w^j \|_{L^2 (\R^N)}^2 \right)$;
\item[(f)] $J(u_n) \to J(u_0) + \sum_{j=1}^k J_\infty (w^j)$.
\end{itemize}
\end{Th}

Before we present the proof, we observe that - compared with standard profile decompositions - the limiting points $w^j$ belong to the $H^1(\R^N)$ space, not merely to $X^1 (\R^N)$. This observation is crucial for the proof, since the energy functional $J_\infty$, associated with $-\Delta w = g(w)$, is defined only on $H^1 (\R^N)$ and is not well-defined for functions in $X^1 (\R^N) \setminus H^1 (\R^N)$. Without this additional regularity, we would not know that each $w^j$ is a critical point of the energy functional, and consequently, we could not ensure that the iterative procedure in the proof terminates after finitely many steps. Here, a crucial role is played by the regularity theorem, which relies on Weyl’s lemma, the Caccioppoli inequality, and continuous embeddings into $H^s (\R^N)$.

\begin{proof}

Since $(u_n) \subset X^1 (\R^N)$ is bounded, it is (up to a subsequence) weakly convergent in $X^1 (\R^N)$ to some $u_0 \in X^1 (\R^N)$. For any $\varphi \in \cC_0^\infty (\R^N)$, $J'(u_0)(\varphi) = J'(u_n)(\varphi) + o(1) = o(1)$ by compact embeddings, so $J'(u_0) = 0$.

Let $v_n^1 := u_n - u_0 \weakto 0$ in $X^1 (\R^N)$. Then $v_n^1 \to 0$ in $L^t_{loc} (\R^N)$ for $t \in [2,2^*)$ and $v_n^1 (x) \to 0$ for a.e. $x \in \R^N$. It is clear that
\begin{equation}\label{E:v_n^1-norm}
    \|v_n^1\|^2_\ell = \| u_n - u_0\|^2_\ell = \|u_n\|^2_\ell + \|u_0\|^2_\ell - 2 \langle u_n, u_0 \rangle_\ell = \|u_n\|^2_\ell - \|u_0\|^2_\ell + o (1).
\end{equation}

Now, we compute
\begin{align*}
\int_{\R^N} F(u_n) - F(v_n^1) \, dx  &= \int_{\R^N} \int_0^1 f(u_n - s u_0)u_0 \, ds \, dx = \int_0^1 \int_{\R^N} f(u_n - s u_0)u_0 \, dx \, ds.
\end{align*}
The family $\{ f(u_n - s u_0)u_0 \}$ is uniformly integrable and tight on $\R^N$, so that from the Vitali convergence theorem
\begin{equation}\label{E:F_convergence}
\begin{split}
\int_{\R^N} F(u_n) - F(v_n^1) \, dx &\to
 \int_0^1 \int_{\R^N} f(u_0 - s u_0)u_0 \, dx \, ds \\
&=   \int_{\R^N} \int_0^1 f(u_0 - s u_0)u_0  \, ds \, dx = \int_{\R^N} F(u_0) \, dx.
\end{split}
\end{equation}
Hence
$$
J(v_n^1) = J(u_n) - J(u_0) + o(1).
$$
Arguing similarly we obtain that $P(v_n^1) = P(u_n)-P(u_0)+o(1)=-P(u_0)+o(1)$.

\textbf{Step 1.}
Suppose that there is $r > 0$ such that
\begin{equation}\label{lionsCond}
\lim_{n \to +\infty} \sup_{z \in \R^N} \int_{B(z, r)} |v_n^1|^2 \, dx = 0.
\end{equation}
From Lemma \ref{lionsLem} there follows that $v_n^1 \to 0$ in $L^t (\R^N)$ for $t \in (2,2^*)$. Since $u_0$ is a critical point of $J$, Proposition \ref{P:Pohozaev} gives $P(u_0) = 0$, hence $P(v_n^1) = P(u_n) - P(u_0) + o(1) = o(1)$. From (F1)--(F3), for every $\varepsilon > 0$ we have
\begin{equation}\label{eq:f-growth}
F(s) \leq \varepsilon s^2 + C_\varepsilon |s|^p, \qquad s \in \R .
\end{equation} 
Using \eqref{eq:f-growth} with $\varepsilon = \frac{\ell}{4}$ we get
$$
o(1) = P(v_n^1) = \xi(v_n^1) + \frac{2^*\ell}{2}\|v_n^1\|_{L^2}^2 - 2^*\int_{\R^N} F(v_n^1)\,dx \geq \xi(v_n^1) + \frac{2^*\ell}{4}\|v_n^1\|_{L^2}^2 - 2^* C_{\ell/4} \|v_n^1\|_{L^p}^p .
$$
Since $\xi \geq 0$ and $\|v_n^1\|_{L^p} \to 0$, we obtain $v_n^1\to 0$, i.e. $u_n \to u_0$ in $X^1 (\R^N)$, and the proof is completed with $k = 0$. Indeed, conditions (b) and (c) are empty, (d) and (e) were just proven with empty sums, and (f) follows from the continuity of $J$.

\textbf{Step 2.}
Suppose that \eqref{lionsCond} does not hold. Then we may find $(y_n^1) \subset \Z^N$ and $n_0$ such that
$$
\inf_{n \geq n_0} \int_{B(y_n^1, r+\sqrt{N})} |v_n^1|^2 \, dx > 0.
$$ 
It is standard to verify that $|y_n^1| \to \infty$. Set $w_n^1 := u_n (\cdot + y_n^1)$. It is clear that $(w_n^1) \subset H^1 (\R^N)$, since $H^1 (\R^N)$ is a translation-invariant subspace.

Note that $0<\frac{N(p-2)}{2p}<1$. Therefore, we can choose $s\in \left( \max \left\{ \frac{N(p-2)}{2p}, \frac12 \right\}, 1 \right)$. Then we have $p<\frac{2N}{N-2s}$. Then $(u_n)$ is bounded in $H^s (\R^N)$, and since $H^s (\R^N)$ is translation-invariant, $(w_n^1) \subset H^s (\R^N)$ is also bounded in $H^s (\R^N)$. Hence, there is $w^1 \in H^s (\R^N)$ such that $w_n^1 \weakto w^1$ in $H^s (\R^N)$, $w_n^1 \to w^1$ in $L^t_{loc} (\R^N)$ for $t \in [2,2^*_s)$ and $w_n^1 (x) \to w^1 (x)$ for a.e. $x \in \R^N$.

Observe that
$$
\int_{B(0, r+\sqrt{N})} |w_n^1|^2 \, dx 
= \int_{B(y_n^1, r+\sqrt{N})} |v_n^1|^2 \, dx + o(1),
$$
and therefore $w^1 \neq 0$. Then, from Lemma \ref{H1-translations-bounded}, for every test function $\varphi \in \cC_0^\infty (\R^N)$
$$
|J'(u_n) (\varphi(\cdot - y_n^1))| \leq \| \nabla J(u_n) \|_\ell \| \varphi (\cdot - y_n^1) \|_\ell \to 0 \cdot \left( \| \nabla \varphi \|_{L^2 (\R^N)}^2 + \ell \| \varphi \|_{L^2 (\R^N)}^2 \right)^{1/2} = 0,
$$
and
\begin{align*}
&\quad o(1) = J'(u_n) (\varphi(\cdot - y_n^1)) \\
&= \int_{\R^N} \nabla u_n \nabla \varphi(\cdot - y_n^1) \, dx - \frac{(N-2)^2}{4} \int_{\R^N} \frac{u_n \varphi(\cdot - y_n^1)}{|x|^2} \, dx - \int_{\R^N} g(u_n) \varphi(\cdot - y_n^1) \, dx = (*).
\end{align*}
Hence, fix a test function $\varphi \in \cC_0^\infty (\R^N)$ and choose $n_0$ such that for $n \geq n_0$ we have\linebreak $\mathrm{dist}\, (\supp \varphi, -y^1_n) > 1$. Then, for $n \geq n_0$ we also have
$$
\int_{\R^N} \frac{u_n \varphi(\cdot - y_n^1)}{|x|^2} \, dx = \int_{\R^N} \frac{u_n(\cdot + y_n^1) \varphi}{|x+y_n^1|^2} \, dx = \int_{\supp \varphi} \frac{u_n(\cdot + y_n^1) \varphi}{|x+y_n^1|^2} \, dx.
$$
Similarly
$$
\int_{\R^N} \nabla u_n \nabla \varphi(\cdot - y_n^1) \, dx = \int_{\R^N} \nabla u_n(\cdot + y_n^1) \nabla \varphi \, dx = \int_{\supp \varphi} \nabla u_n(\cdot + y_n^1) \nabla \varphi \, dx.
$$
Thus
\begin{align*}
\mathclap{
(*) = \int_{\supp \varphi} \nabla u_n(\cdot + y_n^1) \nabla \varphi \, dx - \frac{(N-2)^2}{4} \int_{\supp \varphi} \frac{u_n(\cdot + y_n^1) \varphi}{|x+y_n^1|^2} \, dx - \int_{\supp \varphi} g( u_n(\cdot + y_n^1) ) \varphi \, dx = (**).}
\end{align*}
From Vitali convergence theorem we get
$$
\int_{\supp \varphi} g( u_n(\cdot + y_n^1) ) \varphi \, dx \to \int_{\supp \varphi} g( w^1) \varphi \, dx.
$$
Since $\supp \varphi$ is compact and $|y_n^1| \to \infty$, for every $n$ we can choose $m(n)$ such that $m(n) \to \infty$ and $|x+y_n^1| \geq m(n)$ for $x\in\supp \varphi$. Then
$$
\int_{\supp \varphi} \frac{|u_n(\cdot + y_n^1) \varphi|}{|x+y_n^1|^2} \, dx \leq \frac{1}{m(n)^2} \int_{\supp \varphi} |u_n(\cdot + y_n^1) \varphi| \, dx \leq \frac{1}{m(n)^2} \|u_n\|_{L^2 (\R^N)} \| \varphi\|_{L^2 (\R^N)} \to 0,
$$
since $(u_n)$ is bounded in $L^2 (\R^N)$. Then
\begin{align*}
(**) = \int_{\supp \varphi} \nabla u_n(\cdot + y_n^1) \nabla \varphi \, dx - \int_{\supp \varphi} g( w^1) \varphi \, dx + o(1).
\end{align*}
Hence, we get
$$
\int_{\R^N} \nabla w_n^1 \nabla \varphi \, dx \to \int_{\R^N} g(w^1) \varphi \, dx
$$
for all $\varphi \in \cC_0^\infty (\R^N)$. We can rewrite the left hand side as
$$
\int_{\R^N} \nabla w_n^1 \nabla \varphi \, dx = - \int_{\R^N} w_n^1 \Delta \varphi \, dx
$$
and therefore, since $w_n^1 \to w^1$ in $L^2_{\mathrm{loc}}(\R^N)$, we obtain
$$
- \int_{\R^N} w^1 \Delta \varphi \, dx = \int_{\R^N} g(w^1) \varphi \, dx,
$$
namely $w^1 \in H^s (\R^N)$ is a solution to
$$
-\Delta w^1 = g(w^1) \quad \mbox{in } \cD' (\R^N).
$$
From Theorem \ref{th:usuwanie_osobliwosci}, we obtain that $w^1 \in H^1 (\R^N)$. In particular, $w^1$ is a weak solution to the problem
\begin{equation}\label{eq:limiting}
-\Delta w^1 = g(w^1)
\end{equation}
in $\R^N$. Therefore, by Remark \ref{R:problem_graniczny_wlasnosci}, $J_\infty '(w^1) = 0$ and $P_\infty (w^1) = 0$.

\textbf{Step 3.} Put $v_n^2=u_n-u_0-w^1(\cdot - y_n^1)=v_n^1-w^1(\cdot - y_n^1)$. Note that, by Lemma \ref{H1-translations-bounded} and \eqref{E:v_n^1-norm} we have
\begin{equation*}
\begin{split}
\|v_n^2\|^2_{\ell}&=\|v_n^1\|^2_\ell+\|w^1(\cdot-y_n^1)\|^2_\ell - 2\langle v_n^1,w^1(\cdot-y_n^1)\rangle_\ell
\\
&=\|u_n\|^2_\ell-\|u_0\|^2_\ell+o(1)+\|w^1(\cdot -y_n^1)\|^2_\ell
-2\langle u_n,w^1(\cdot -y_n^1)\rangle_\ell
+2\langle u_0,w^1(\cdot -y_n^1)\rangle_\ell\\
&=\|u_n\|^2_\ell-\|u_0\|^2_\ell+\left( \|\nabla w^1\|_{L^2 (\R^N)}^2 + \ell \| w^1 \|_{L^2 (\R^N)}^2 \right) - 2\langle u_n,w^1(\cdot -y_n^1)\rangle_\ell +o(1)=(*),
\end{split}
\end{equation*}
where the term $2\langle u_0,w^1 (\cdot - y_n^1)\rangle_\ell $ disappears by Lemma \ref{L:scalar_product_translation}. Then, we compute
\begin{align*}
\langle u_n, w^1 (\cdot - y_n^1) \rangle_\ell = &\int_{\R^N} \nabla u_n(\cdot + y_n^1) \nabla w^1 \,dx + \ell \int_{\R^N}  u_n(\cdot + y_n^1)  w^1 \,dx \\ &- \frac{(N-2)^2}{4} \int_{\R^N} \frac{u_n w^1 (\cdot - y_n^1)}{|x|^2} \, dx.
\end{align*}
We will show that
\begin{equation}\label{E:temp-new}
\int_{\R^N} \frac{u_n w^1 (\cdot - y_n^1)}{|x|^2} \, dx\to 0
\end{equation}
to finally get
\[
\begin{split}
(*)=&\|u_n\|^2_\ell-\|u_0\|^2_\ell+\left( \|\nabla w^1\|_{L^2 (\R^N)}^2 + \ell \| w^1 \|_{L^2 (\R^N)}^2 \right)\\
&-2 \left( \int_{\R^N} \nabla u_n(\cdot + y_n^1) \nabla w^1 \,dx + \ell \int_{\R^N}  u_n(\cdot + y_n^1)  w^1 \,dx  \right) +o(1) = (**).
\end{split}
\]
Now, using the fact that $w^1$ is a solution to \eqref{eq:limiting}, we obtain
\begin{align}\label{E:temp-new2}
\int_{\R^N} \nabla u_n(\cdot + y_n^1) \nabla w^1 \,dx + \ell \int_{\R^N}  u_n(\cdot + y_n^1)  w^1 \,dx = \int_{\R^N} f(w^1) w_n^1 \, dx.
\end{align}
Since $w_n^1$ is bounded in $L^2 (\R^N)$ and in $L^p (\R^N)$, $w_n^1 \to w^1$ in $L^2_{\mathrm{loc}} (\R^N)$ and in $L^p_{\mathrm{loc}} (\R^N)$, from Vitali convergence theorem
$$
\int_{\R^N} f(w^1) w_n^1 \, dx \to \int_{\R^N} f(w^1) w^1 \, dx = \|\nabla w^1\|_{L^2 (\R^N)}^2 + \ell \| w^1 \|_{L^2 (\R^N)}^2.
$$
Thus
\begin{align*}
(**) &= \|u_n\|^2_\ell-\|u_0\|^2_\ell - \left( \|\nabla w^1\|_{L^2 (\R^N)}^2 + \ell \| w^1 \|_{L^2 (\R^N)}^2 \right) + o(1).
\end{align*}
To prove \eqref{E:temp-new} take any $R>0$ and consider
\[
\int_{\R^N} \frac{u_n w^1 (\cdot - y_n^1)}{|x|^2} \, dx
=\int_{B(0,R)} \frac{u_n w^1 (\cdot - y_n^1)}{|x|^2} \, dx
+\int_{\R^N\setminus B(0,R)} \frac{u_n w^1 (\cdot - y_n^1)}{|x|^2} \, dx=: I_1+I_2
\]
By H\"{o}lder inequality
\[
|I_1|\leq \|u_n\|_{L^{\frac{2N}{N-2s}}(\R^N)} \cdot 
\left(\int_{B(0,R)} \frac{|w^1(x-y_n)|^{\frac{2N}{N+2s}}}{|x|^{\frac{4N}{N+2s}}} \, dx\right)^{\frac{N+2s}{2N}}.
\]
By the continuity of the embedding $X^1(\R^N)\subset H^s(\R^N)$, the sequence $(u_n)$ is bounded in $L^{\frac{2N}{N-2s}} (\R^N)$. 

Since $w^1 \in H^1 (\R^N)$ is a weak solution to \eqref{eq:limiting}, from \cite[Theorem 2]{PankovDecay}, there is $\alpha > 0$ such that $|w^1(x)| \lesssim e^{-\alpha |x|}$. For $x\in B(0,R)$ we have $|x-y_n|\geq |y_n|-|x|\geq |y_n|-R\to\infty$, so one has
     \[
     \delta_n = \sup_{x\in B(0,R)} |w^1(x-y_n)| \lesssim  e^{-\alpha (|y_n|-R)} \to 0 \quad \mbox{as } n\to\infty.
     \]
Moreover, since $s > \frac12$,
$$
\frac{4N}{N+2s} < N.
$$
Thus $\int_{B(0,R)} |x|^{- \frac{4N}{N+2s}} \, dx < +\infty$. Hence
$$
|I_1| \lesssim \delta_n \to 0 \quad \mbox{as } n \to \infty.
$$
Since $u_n, w^1\in L^2 (\R^N)$, and $(u_n)$ is bounded in $L^2 (\R^N)$, we have
\[
|I_2| \leq \frac{1}{R^2} \|u_n\|_{L^2(\R^N)}\|w^1\|_{L^2(\R^N)}.
\]
Thus, for any $R > 0$, we get
$$
\limsup_{n\to\infty} \left| \int_{\R^N} \frac{u_n w^1 (\cdot - y_n^1)}{|x|^2} \, dx \right| \lesssim \frac{1}{R^2}.
$$
Hence,
$$
\int_{\R^N} \frac{u_n w^1 (\cdot - y_n^1)}{|x|^2} \, dx \to 0
$$
and the proof of \eqref{E:temp-new} is completed.

Arguing as in \eqref{E:F_convergence} for the pair of functions $(v_n^1, w^1)$ instead of $(u_n, u_0)$, and then directly from \eqref{E:F_convergence} we get
\begin{equation}\label{E:nonlinearity_decomposition}
\begin{split}
\int_{\R^N} F(v_n^2) \, dx 
&=\int_{\R^N} F(v_n^2(\cdot + y_n^1)) \, dx \\
&=\int_{\R^N} F(v_n^1(\cdot+y_n^1))\,dx -\int_{\R^N} F(w^1)\,dx +o(1)\\
&=\int_{\R^N} F(v_n^1)\,dx -\int_{\R^N} F(w^1)\,dx +o(1)\\
&=\int_{\R^N} F(u_n)\,dx -\int_{\R^N} F(u_0)\,dx - \int_{\R^N} F(w^1)\,dx +o(1).
\end{split}
\end{equation}
Combining the foregoing with $\|v_n^2\|_\ell^2=\|u_n\|_\ell^2 - \|u_0\|_\ell^2 - \left( \|\nabla w^1\|_{L^2 (\R^N)}^2 + \ell \| w^1 \|_{L^2 (\R^N)}^2 \right) +o(1)$ we obtain
\[
J(v_n^2)=J(u_n)-J(u_0)-J_\infty(w^1)+o(1)
\]
and, by Remark \ref{R:problem_graniczny_wlasnosci},
\[
P(v_n^2)=P(u_n)-P(u_0)-P_\infty(w^1)+o(1)=-P(u_0)+o(1).
\]
Consider 
\[
\Theta_2=\lim_{n \to +\infty} \sup_{z \in \R^N} \int_{B(z, r)} |v_n^2|^2 \, dx
\]
for a fixed radius $r>0$.

If $\Theta_2=0$, then $v_n^2 \to 0$ in $L^p(\R^N)$ by Lemma \ref{lionsLem}, and $P(v_n^2) = -P(u_0) + o(1) = o(1)$ by Proposition \ref{P:Pohozaev}; the argument of Step 1 gives $v_n^2\to 0$ in $X^1(\R^N)$ and the proof is completed with $k=1$.

Suppose $\Theta_2>0$. Then we reason as in Step 2. We find the sequence $(y_n^2)\subset \Z^N$, $|y_n^2|\to\infty$ such that 
$$
\inf_{n \geq n_0} \int_{B(y_n^2, r+\sqrt{N})} |v_n^2|^2 \, dx > 0.
$$ 
It is classical to check that $|y_n^1-y_n^2|\to\infty$. We put $w_n^2=u_n(\cdot+y_n^2)\in H^1(\R^N)$. Then we find $w^2\in H^1(\R^N)$ such that 
\begin{itemize}
\item $w_n^2\weakto w^2$ in $H^s(\R^N)$ for $s$ sufficiently close to $1$,
\item $w_n^2\to w^2$ in $L^t_{loc}(\R^N)$ for $t\in[2,2^*_s)$,
\item $w^2$ is a weak solution of $-\Delta w^2 =g(w^2)$ on $\R^N$,
\item $J'_\infty(w^2)=0, \ P_\infty(w^2)=0$.
\end{itemize}
Put 
\[
v_n^3:=v_n^2-w^2(\cdot - y_n^2)
=u_n-u_0-w^1(\cdot - y_n^1)-w^2(\cdot - y_n^2).
\]
The same way as in \eqref{E:nonlinearity_decomposition} we get
\[
\int_{\R^N} F(v_n^3) \, dx 
=\int_{\R^N} F(u_n)\,dx 
-\int_{\R^N} F(u_0)\,dx 
- \int_{\R^N} F(w^1)\,dx 
- \int_{\R^N} F(w^2)\,dx +o(1).
\]

Moreover, by Lemma \ref{L:scalar_product_translation}, Lemma \ref{H1-translations-bounded}, and the arguments leading to \eqref{E:temp-new} and \eqref{E:temp-new2} (with $y_n^1$, $w^1$, $w_n^1$ replaced by $y_n^2$, $w^2$, $w_n^2$),
\[
\begin{split}
\|v_n^3\|^2_\ell=&\|v_n^2\|^2_\ell + \|w^2(\cdot - y_n^2)\|^2_\ell
-2\langle v_n^2, w^2(\cdot - y_n^2)\rangle_\ell\\
=& \|v_n^2\|^2_\ell
+\left( \|\nabla w^2\|_{L^2 (\R^N)}^2 + \ell \| w^2 \|_{L^2 (\R^N)}^2 \right)\\
&-2\left( \langle u_n,w^2(\cdot - y_n^2)\rangle_\ell
-\langle u_0,w^2(\cdot - y_n^2)\rangle_\ell
-\langle w^1(\cdot - y_n^1),w^2(\cdot - y_n^2)\rangle_\ell \right) + o(1)
\\
=&\|v_n^2\|^2_\ell
-\left( \|\nabla w^2\|_{L^2 (\R^N)}^2 + \ell \| w^2 \|_{L^2 (\R^N)}^2 \right)+2\langle w^1(\cdot - y_n^1),w^2(\cdot - y_n^2)\rangle_\ell+o(1).
\end{split}
\]
By the Cauchy-Schwarz inequality and Lemma \ref{L:MederskiGuo} we have
\[
\left| \int_{\R^N} \frac{ w^1(\cdot - y_n^1)w^2(\cdot - y_n^2)}{|x|^2} \, dx \right| \leq \left(
\int_{\R^N} \frac{ (w^1(\cdot - y_n^1))^2}{|x|^2} \, dx \right)^{1/2} \left(
\int_{\R^N} \frac{ (w^2(\cdot - y_n^2))^2}{|x|^2} \, dx \right)^{1/2}\to 0 .
\]
Therefore, utilizing $|y_n^1-y_n^2|\to\infty$, we get
\[
\langle w^1(\cdot - y_n^1),w^2(\cdot - y_n^2)\rangle_\ell=\int_{\R^N} \nabla w^1 \cdot \nabla w^2(\cdot - (y_n^2-y_n^1)) + \ell\, w^1 w^2(\cdot - (y_n^2-y_n^1)) \, dx + o(1)\to 0.
\]

Hence,
\[
\begin{split}
\|v_n^3\|^2_\ell&=\|v_n^2\|^2_\ell 
-\left( \|\nabla w^2\|_{L^2 (\R^N)}^2 + \ell \| w^2 \|_{L^2 (\R^N)}^2 \right)+o(1)\\
&=\|u_n\|^2_\ell-\|u_0\|^2_\ell-
\sum_{j=1}^2 \left( \|\nabla w^j\|_{L^2 (\R^N)}^2 + \ell \| w^j \|_{L^2 (\R^N)}^2 \right)+o(1).
\end{split}
\]
Finally we can write
\[
J(v_n^3)=J(u_n)-J(u_0)-\sum_{j=1}^2 J_\infty(w^j)+o(1)
\]
and
\[
P(v_n^3)=P(u_n)-P(u_0)-\sum_{j=1}^2 P_\infty(w^j)+o(1)=-P(u_0)+o(1).
\]

\textbf{Step 4.} We can repeat the above procedure to construct, for some $k\in\N$:
\begin{itemize}
\item the sequences $(y_n^j)\subset \Z^N$, $j=1,\ldots,k$ such that $|y_n^j|\to\infty$ and $|y_n^j-y_n^{j'}|\to \infty$ for $j\neq j'$,
\item $w^j\in H^1(\R^N)$, $j=1,\ldots,k$, such that $P_\infty(w^j)=0$,
\item $\|u_n\|^2_\ell\to \|u_0\|^2_\ell+\sum_{j=1}^k  \left( \|\nabla w^j\|_{L^2 (\R^N)}^2 + \ell \| w^j \|_{L^2 (\R^N)}^2 \right)$.
\end{itemize}
Note that, from (F1)--(F3), for every $\varepsilon > 0$, there is a constant $C_\varepsilon > 0$ such that
\begin{align*}
\| \nabla w^j \|_{L^2(\R^N)}^2 &+ \ell \| w^j \|_{L^2(\R^N)}^2 = \int_{\R^N} f(w^j)w^j \, dx \leq \varepsilon \|w^j\|_{L^2(\R^N)}^2 + C_\varepsilon \|w^j\|_{L^p(\R^N)}^p \\
&\leq C \varepsilon \left( \| \nabla w^j \|_{L^2(\R^N)}^2 + \ell \| w^j \|_{L^2(\R^N)}^2 \right) + C C_\varepsilon \left( \| \nabla w^j \|_{L^2(\R^N)}^2 + \ell \| w^j \|_{L^2(\R^N)}^2 \right)^{p/2},
\end{align*}
where we used that $J_\infty '(w^j)(w^j) = 0$ and $C > 0$ comes from Sobolev embeddings. Hence, choosing sufficiently small $\varepsilon > 0$,
$$
\frac{1 - C\varepsilon}{C C_\varepsilon} \leq \left(\| \nabla w^j \|_{L^2 (\R^N)}^2 + \ell \| w^j \|_{L^2(\R^N)}^2 \right)^{(p-2)/2}.
$$
It means that for every $j$,
$$
\| \nabla w^j \|_{L^2(\R^N)}^2 + \ell \| w^j \|_{L^2(\R^N)}^2 \gtrsim 1,
$$
hence the procedure needs to terminate after a finite number of steps, since $(u_n)$ is bounded in $X^1(\R^N)$.

\end{proof}

\section{Existence of a solution}\label{sect:5}

We will start by showing that the Poho\v{z}aev-Palais-Smale sequences are bounded. 

\begin{Lem}\label{Lem:PPS-bounded}
Suppose that $c > 0$ and $(u_n) \subset H^1 (\R^N)$ is a sequence such that $J(u_n) \to c$, $J'(u_n) \to 0$ and $P(u_n) \to 0$. Then $(u_n)$ is bounded in $X^1 (\R^N)$.
\end{Lem}

\begin{proof}
Since $J(u_n) \to c$, and using $P(u_n) \to 0$, we get
$$
\frac12 \xi (u_n) =  \int_{\R^N} G(u_n) \, dx + c + o(1) = \frac{1}{2^*} \xi(u_n) + c + o(1)
$$
and therefore
$$
\left(\frac12 - \frac{1}{2^*}\right) \xi(u_n) = c + o(1). 
$$
In particular, $(\xi(u_n))$ is bounded. Now, what remains is to show that $\|u_n\|_{L^2 (\R^N)}$ is bounded. Using Hardy-Gagliardo-Nirenberg inequality (\cite[Theorem 2]{Mukherjee}) and \eqref{eq:f-growth} with $\varepsilon = \frac{\ell}{4}$
\begin{align*}
\frac{\ell}{2}\|u_n\|_{L^2}^2 &= J(u_n) - \frac12 \xi(u_n) + \int_{\R^N} F(u_n)\,dx \leq c + \frac{\ell}{4}\|u_n\|_{L^2}^2 + C \|u_n\|_{L^p}^{p} + o(1) \\
&\leq c + \frac{\ell}{4}\|u_n\|_{L^2}^2 + C \xi(u_n)^{\frac{N(p-2)}{4}} \|u_n\|_{L^2}^{p - \frac{N(p-2)}{2}} + o(1).
\end{align*}
Since $p - \frac{N(p-2)}{2} < 2$ we obtain that $\|u_n\|_{L^2 (\R^N)}$ is bounded.

\end{proof}

Now we are ready to prove the existence of a nontrivial solution.

\begin{proof}[Proof of Theorem \ref{Th:Main1-Existence}(a)]

Define $c_\infty := \inf_{\gamma \in \Gamma_\infty} \sup_{t \in [0,1]} J_\infty(\gamma(t))$, where
$$
\Gamma_\infty := \left\{ \gamma \in C([0,1]; H^1 (\R^N)) \ : \ \gamma(0) = 0, \ J_\infty(\gamma(1)) < 0 \right\}.
$$
It is known that $c_\infty > 0$ and that there is an optimal path $\gamma_\infty \in \Gamma_\infty$ such that
$$
c_\infty = \sup_{t\in [0,1]} J_\infty (\gamma_\infty(t)),
$$
see \cite{Lions2, BJM2009}. Moreover
$$
c_\infty = \inf \left\{ J_\infty(w) \ : \ w \in H^1 (\R^N) \setminus \{0\}, \ J_\infty'(w) = 0 \right\}.
$$
Since $J(u) < J_\infty(u)$ for $u \in H^1 (\R^N) \setminus \{0\}$ and $\Gamma_\infty \subset \Gamma$, we obtain the following
\begin{align*}
c_\infty = \sup_{t\in [0,1]} J_\infty (\gamma_\infty(t)) > \sup_{t\in [0,1]} J (\gamma_\infty(t)) \geq \inf_{\gamma \in \Gamma} \sup_{t\in[0,1]} J (\gamma(t)) = c,
\end{align*}
where $c$ and $\Gamma$ are given by \eqref{def:c} and \eqref{def:Gamma}, respectively.

Let $(u_n)$ be a Poho\v{z}aev-Palais-Smale sequence obtained in Lemma \ref{Lem:ExistencePPS}. From Lemma \ref{Lem:PPS-bounded} this sequence is bounded in $X^1 (\R^N)$ and, up to choosing a subsequence, (a)--(f) of Theorem \ref{Th:splitting} hold. Since $w^j \in H^1 (\R^N)$ are nontrivial critical points of $J_\infty$, we know that $J_\infty(w^j) \geq c_\infty$. Hence, from (f),
$$
J(u_n) = J(u_0) + \sum_{j=1}^k J_\infty (w^j) + o(1) \geq J(u_0) + k c_\infty + o(1).
$$
If $k \geq 1$, we have
$$
c \geq J(u_0) + k c_\infty > J(u_0) + k c.
$$
Since $u_0$ is a critical point of $J$ satisfying $P(u_0) = 0$, $J(u_0) = \left( \frac12 - \frac{1}{2^*} \right) \xi(u_0) \geq 0$, we obtain a contradiction. Therefore $k = 0$. Then, from (d) and (e), $u_n \to u_0$ in $X^1 (\R^N)$, $J(u_0) = c$ and $u_0$ is a nontrivial solution to \eqref{eq:main}. Moreover, the solution $u_0$ satisfies $P(u_0) = 0$.
\end{proof}

\section{Further properties of a solution}\label{sect:6}

\subsection{\texorpdfstring{$u_0$}{u0} is a minimizer on Poho\v{z}aev constraint}

Recall the Poho\v{z}aev constraint
$$
\cM := \left\{ u \in X^1 (\R^N) \setminus \{0\} \ : \ \xi(u) = 2^* \int_{\R^N} G(u) \, dx \right\}.
$$
We will show that the solution $u_0$ is a minimizer on $\cM$, namely that there holds
$$
J(u_0) = \inf_{\cM} J.
$$

\begin{proof}[Proof of Theorem \ref{Th:Main1-Existence}(b)]
Since $u_0 \in \cM$, clearly $c = J(u_0) \geq m := \inf_{\cM} J$. Hence it is enough to show that $m \geq c$. For that purpose, take any $u \in \cM$ and consider the rescaled function $U(\theta) := u(\cdot / \theta)$ for $\theta \geq 0$, where we set $U(0) = 0$. Then
$$
\| U(\theta) \|_\ell^2 = \theta^{N-2} \xi(u) + \theta^N \ell \|u\|_{L^2(\R^N)}^2
$$
and
$$
\Psi(\theta) := J(U(\theta)) = \frac{\theta^{N-2}}{2} \xi(u) - \theta^N \int_{\R^N} G(u) \, dx.
$$
Then, clearly $U : [0,\infty) \rightarrow X^1 (\R^N)$ is continuous. Since $u \in \cM$, we get
$$
\int_{\R^N} G(u) \, dx = \frac{1}{2^*} \xi(u) > 0.
$$
The last inequality follows from Lemma \ref{L:groundstate}, since $u \neq 0$. Now, we consider $\frac{d}{d \theta} \Psi$. Note that
$$
\frac{d}{d \theta} \Psi(\theta) = \frac{(N-2) \theta^{N-3}}{2 } \xi(u) - N \theta^{N-1} \int_{\R^N} G(u) \, dx = \theta^{N-3} \left( \frac{N-2}{2 } \xi(u) - N \theta^{2} \int_{\R^N} G(u) \, dx \right).
$$
Since $u \in \cM$, we get $\frac{d}{d \theta} \Psi(1) = 0$, and clearly $\frac{d}{d \theta} \Psi(\theta) > 0$ for $\theta \in (0,1)$ and $\frac{d}{d \theta} \Psi(\theta) < 0$ for $\theta \in (1,\infty)$.

Hence, the following properties are satisfied by $U$,
$$
U(0) = 0, \ J(U(\vartheta)) < 0, \ U(1) = u \in U([0,\vartheta])
$$
for sufficiently large $\vartheta > 1$; moreover $\max_{\theta \in [0,\vartheta]} J(U(\theta)) = J(u)$. In particular, the path 
\begin{equation}\label{eq:path}
\gamma := U(\vartheta \cdot)
\end{equation}
belongs to $\Gamma$, where $\Gamma$ is given by \eqref{def:Gamma}. Hence $c \leq J(u)$. Since $u \in \cM$ was arbitrary, $c \leq m$; moreover every nontrivial solution lies in $\cM$, and the proof is completed.
\end{proof}

\subsection{Regularity of the solution}

\begin{proof}[Proof of Theorem \ref{Th:Main1-Existence}(c)]
Let $u\in X^1(\R^N)$ be a weak solution of \eqref{eq:main}. Fix $x_0\in\R^N\setminus\{0\}$ and its bounded open neighborhood $\Omega\subset\overline{\Omega}\subset \R^N \setminus \{0\}$. Testing $J'(u)=0$ on the function $\varphi\in C_0^\infty(\Omega)$ we easily get that $u\in H^1(\Omega)$ is a weak solution of
\begin{equation}\label{E:C_2_temp}
    -\Delta u=g(u)+\frac{(N-2)^2}{4|x|^2} u \quad\text{in } \Omega.
\end{equation}
Observe that 
\[
\left|g(u)+\frac{(N-2)^2}{4|x|^2} u \right|\lesssim |u|+|u|^{p-1} \leq (1+|u|^{p-2})(1+|u|)
\]
and $|u|^{p-2}\in L^{N/2}(\Omega)$. Therefore, by Brezis-Kato lemma (see \cite{BrezisKato} and \cite[Lemma B.3]{Struwe}), $u\in L^q_{\loc}(\Omega)$ for any $q<\infty$. Then, by Calder\'{o}n-Zygmund inequality (see \cite[Lemma B.2]{Struwe}), $u\in W^{2,q}_{\loc}(\Omega)$ and next, by Sobolev embedding theorem, $u\in C^{1,\alpha}_{\loc}(\Omega)$. Since $x_0$ was an arbitrary point, we conclude $u\in C^{1,\alpha}_{\loc}(\R^N\setminus\{0\})$.
\end{proof}

\begin{proof}[Proof of Theorem \ref{Th:Main1-Existence}(d)]
We have already shown that $u\in C^{1,\alpha}_{\loc}(\R^N \setminus \{0\})$. Since $g$ is H\"older continuous, by Schauder estimates we get $u\in C^{2}(\R^N \setminus \{0\})$.
\end{proof}

\subsection{There is a radial solution, which is a minimizer on \texorpdfstring{$\cM$}{M}}

For a nonnegative, measurable function $u$, let $u^*$ denote its radially decreasing rearrangement (also known as the Schwarz symmetrization). The well-known Pólya-Szegő inequality states that if $u \in H^1 (\R^N)$, then $u^* \in H^1 (\R^N)$ and
$$
\int_{\R^N} |\nabla u^*|^2 \, dx \leq \int_{\R^N} |\nabla u|^2 \, dx.
$$
Moreover, the following properties hold (see \cite{LiebLoss})
$$
\int_{\R^N} \frac{|u^*|^2}{|x|^2} \, dx \geq \int_{\R^N} \frac{u^2}{|x|^2} \, dx, \quad \int_{\R^N} G(u^*) \, dx = \int_{\R^N} G(u) \, dx.
$$
Observe that these inequalities imply that
$$
\xi(u^*) \leq \xi(u), \quad J(u^*) \leq J(u)
$$
as long as $u \in H^1 (\R^N)$. We will show that it also extends to the space $X^1 (\R^N)$.
\begin{Th}[Pólya-Szegő inequality in $X^1 (\R^N)$]
For any $u \in X^1 (\R^N)$, $u \geq 0$, there holds
$$
\xi(u^*) \leq \xi(u).
$$
\end{Th}

\begin{proof}
Take any $u \in X^1 (\R^N) \setminus H^1 (\R^N)$, $u \geq 0$ and take a sequence $u_n \in H^1 (\R^N)$, $u_n \geq 0$ such that $u_n \to u$ in $X^1 (\R^N)$. Consider $v_n := (u_n)^*$. Then, clearly
$$
\xi(v_n) \leq \xi(u_n), \quad \| v_n \| \leq \| u_n \|.
$$
Since $u_n \to u$ in $X^1(\R^N)$, $(v_n)$ is bounded in $X^1 (\R^N)$ and therefore, up to a subsequence, $v_n \weakto w$ in $X^1 (\R^N)$ and $v_n \to w$ in $L^2_{\mathrm{loc}}(\R^N)$ for some $w \in X^1 (\R^N)$. Since $v_n \to u^*$ in $L^2 (\R^N)$, from the uniqueness of the limit in $L^2 (\R^N)$, $w = u^*$ almost everywhere. Hence $u^* \in X^1 (\R^N)$. Since $\| \cdot \|$ is weakly lower semicontinuous,
$$
\| u^* \|^2 \leq \liminf_{n\to\infty} \|v_n\|^2 \leq \liminf_{n\to\infty} \|u_n\|^2 = \|u\|^2.  
$$
Taking into account that $\| u^* \|_{L^2 (\R^N)} = \| u \|_{L^2 (\R^N)}$, we obtain that
$$
\xi(u^*) \leq \xi(u).
$$
\end{proof}

\begin{proof}[Proof of Theorem \ref{Th:Main1-Existence}(e)]
We assume in addition that $g$ is odd. Let $u_0 \in X^1 (\R^N)$ be the solution that minimizes $J$ on $\cM$, and let $\gamma$ denote the path given by \eqref{eq:path} associated with $u_0$. Then we consider the Schwarz symmetrization $|\gamma (t)|^*$ of $|\gamma(t)|$ for every $t \in [0,1]$. Then
$$
|\gamma (0)|^* = 0^* = 0, \ J(|\gamma(1)|^*) \leq J(|\gamma(1)|) \leq J(\gamma(1)) < 0
$$
and therefore $|\gamma (\cdot )|^* \in \Gamma$, where $\Gamma$ is given by \eqref{def:Gamma}. Hence
$$
c \leq \max_{t \in [0,1]} J (|\gamma (t)|^*).
$$
On the other hand
$$
\max_{t \in [0,1]} J (|\gamma (t)|^*) \leq \max_{t \in [0,1]} J (|\gamma (t)|)\leq \max_{t \in [0,1]} J(\gamma(t)) = J(u_0) = c.
$$
and
$$
J (|\gamma (t)|^*) \leq J(|\gamma(t)|) \leq J(\gamma(t)) < c
$$
for $t \neq \frac{1}{\vartheta}$, where $\vartheta$ is the same as in \eqref{eq:path}. Thus
$$
J (\gamma^* (1/\vartheta) ) = c.
$$
Define $v_0 := \gamma^* (1/\vartheta)$. Then, from Lemma \ref{lemma:critical_point_on_path}, $v_0$ is a critical point of $J$ and $J(v_0) = c$. To see that $v_0 \in \cM$ observe that
$$
v_0 = \gamma^*(1/\vartheta) = |u_0|^*
$$
and from $J(v_0) = J(u_0)$ we immediately obtain $\xi(v_0) = \xi(u_0)$. Hence
$$
\xi(v_0) = \xi(u_0) = 2^* \int_{\R^N} G(u_0) \, dx = 2^* \int_{\R^N} G(v_0) \, dx
$$
and $v_0 \in \cM$. From Theorem \ref{Th:Main1-Existence}(c), $v_0 \in C^{1,\alpha}_{\loc}(\R^N \setminus \{0\})$ for some $\alpha \in (0,1)$. Since $v_0$ is a radial function (with a small abuse of notation, $v_0(x)=v_0(r)$, $r=|x|$) we can rewrite \eqref{eq:main}, in a weak sense, as
\[
-v_0''-\frac{N-1}{r} v_0'=g(v_0)+\frac{(N-2)^2}{4r^2} v_0 \
\Longleftrightarrow \ -\left(v_0'r^{N-1}\right)'=r^{N-1}\left(g(v_0)+\frac{(N-2)^2}{4r^2} v_0\right).
\]
The right-hand side is continuous in $r$, hence by the standard bootstrap argument the map $r \mapsto v_0'(r)r^{N-1}\in C^1$ on its domain, and therefore $v_0\in C^2$ in a neighborhood of $x_0$. Since $x_0$ was an arbitrary point, we conclude $v_0\in C^2(\R^N\setminus\{0\})$.
\end{proof}

\subsection{Nonnegative solutions do not belong to \texorpdfstring{$H^1(\R^N)$}{H1}}

\begin{proof}[{Proof of Theorem \ref{Th:Main1-Existence}(f)}]
Let $u \in X^1 (\R^N) \cap C^2 (\R^N \setminus \{0\})$ be a nonnegative solution to \eqref{eq:main}. As in the argument in the proof of \cite[Theorem 1.1(iv)]{LiLiTang}, we introduce
$$
u(x) = |x|^\tau w(x),
$$
where $\tau\in\left(-\frac{N-2}2,0\right)$. Then, $w$ solves
$$
- \div \left( |x|^{2\tau} \nabla w \right) = |x|^\tau \left( \frac{\left(\tau + \frac{N-2}{2} \right)^2}{|x|^2} |x|^\tau w + g(|x|^\tau w) \right).
$$
By (G1)--(G3) and $g(s) \geq 0$ for sufficiently large $s$, we have $g(s) \gtrsim -s$ for all $s \geq 0$. Hence we find a radius $\rho > 0$ such that
$$
\frac{\left(\tau + \frac{N-2}{2} \right)^2}{|x|^2} |x|^\tau w + g(|x|^\tau w)  \geq 0 \quad \mbox{for} \ 0 < |x| < \rho. 
$$
Then
$$
- \div \left( |x|^{2\tau} \nabla w \right) \geq 0 \quad \mbox{in } B(0,\rho) \setminus \{0\}.
$$
Repeating the argument in \cite{LiLiTang}, 
$$
u(x) = |x|^\tau w(x) \gtrsim |x|^\tau \quad \mbox{in a neighborhood of the origin},
$$
i.e. $u(x)\to \infty$ for $x\to 0$. Therefore, we can choose a radius $\widetilde{\rho} > 0$ such that $u(x) > 0$ and $g(u(x)) \geq 0$ in $B(0,\widetilde{\rho}) \setminus \{0\}$. Hence, $u$ satisfies the inequality
$$
-\Delta u - \frac{(N-2)^2}{4 |x|^2} u \geq 0 \quad \mbox{in } B(0,\widetilde{\rho}) \setminus \{0\}.
$$
From \cite[Proposition 2.1]{Smets}, $u \not\in H^1 (B(0,\widetilde{\rho}))$.
\end{proof}

\section{Existence of non-radial solutions}\label{sect:7}

In this section we will prove Theorem \ref{Th:Main2}. Following \cite{MR4173560, BartschWillem}, we recall the subspace $X_\tau$ given by \eqref{eq:Xtau} and the group $\mathcal{O} = \mathcal{O}(M) \times \mathcal{O}(M) \times \mathrm{id}_{N-2M}$ acting on $\R^N = \R^M \times \R^M \times \R^{N-2M}$, $N \geq 4$. 

\begin{Lem}\label{L:Xtau}
For every $u \in X_\tau$ we have $u \in H^1(\R^N)$ and
\begin{equation}\label{eq:Xtau-coercive}
\frac{4(N-1)}{N^2}\, \|\nabla u\|_{L^2(\R^N)}^2 \leq \xi(u) \leq \|\nabla u\|_{L^2(\R^N)}^2 .
\end{equation}
In particular, the norms of $X^1(\R^N)$ and $H^1(\R^N)$ are equivalent on $X_\tau$, $X_\tau$ is a closed subspace of $H^1(\R^N)$, and $X_\tau \cap X^1_{\mathcal{O}}(\R^N) \subset H^1(\R^N)$.
\end{Lem}

\begin{proof}
 Let $u \in \cC_0^\infty(\R^N)$ satisfy $u \circ \tau = -u$. Since $\tau$ is a linear isometry of $\R^N$, it preserves the surface measure on $\mathbb{S}^{N-1}$, hence for every $r > 0$
$$
\int_{\mathbb{S}^{N-1}} u(r\theta)\,d\theta = \int_{\mathbb{S}^{N-1}} u(r\tau\theta)\,d\theta = -\int_{\mathbb{S}^{N-1}} u(r\theta)\,d\theta ,
$$
i.e. $u(r\,\cdot)$ has zero mean on $\mathbb{S}^{N-1}$. Denote by $\nabla_\theta$ the tangential gradient on $\mathbb{S}^{N-1}$, so that $|\nabla u(r\theta)|^2 = |\partial_r u(r\theta)|^2 + r^{-2}|\nabla_\theta u(r\theta)|^2$. Since the first nonzero eigenvalue of the Laplace-Beltrami operator on $\mathbb{S}^{N-1}$ equals $N-1$, the Poincar\'e inequality on the sphere gives
$$
\int_{\mathbb{S}^{N-1}} |\nabla_\theta u(r\theta)|^2 \, d\theta \geq (N-1) \int_{\mathbb{S}^{N-1}} u(r\theta)^2 \, d\theta, \qquad r > 0 .
$$
Moreover, for $\phi \in \cC^1_0([0,\infty))$ we have the following Hardy-type inequality
$$
\int_0^\infty |\phi'(r)|^2 r^{N-1}\,dr \geq \left( \frac{N-2}{2} \right)^2 \int_0^\infty \phi(r)^2 r^{N-3}\,dr .
$$
Indeed, since $N \geq 3$, integrating by parts and using the Cauchy-Schwarz inequality,
\begin{align*}
\int_0^\infty \phi^2 r^{N-3}\,dr &= -\frac{2}{N-2} \int_0^\infty \phi\phi' r^{N-2}\,dr \\
&\leq \frac{2}{N-2} \left(\int_0^\infty \phi^2 r^{N-3}\,dr\right)^{1/2} \left(\int_0^\infty |\phi'|^2 r^{N-1}\,dr\right)^{1/2}.
\end{align*}
Applying it to $\phi = u(\cdot\,\theta)$ for every $\theta \in \mathbb{S}^{N-1}$, we obtain
\begin{align*}
\int_{\R^N} |\nabla u|^2\,dx &= \int_0^\infty \int_{\mathbb{S}^{N-1}} \left( |\partial_r u|^2 + \frac{|\nabla_\theta u|^2}{r^2}\right) r^{N-1}\,d\theta\,dr \geq  \frac{N^2}{4} \int_{\R^N} \frac{u^2}{|x|^2}\,dx .
\end{align*}
Hence
$$
\xi(u) = \int_{\R^N} |\nabla u|^2\,dx - \left(\frac{N-2}{2}\right)^2 \int_{\R^N} \frac{u^2}{|x|^2}\,dx \geq \frac{4(N-1)}{N^2} \int_{\R^N} |\nabla u|^2\,dx ,
$$
which is the first inequality in \eqref{eq:Xtau-coercive}. The second inequality in \eqref{eq:Xtau-coercive} is the usual Hardy inequality.

For the general case, note that the map $v \mapsto v\circ\tau$ is a linear isometry of $H^1(\R^N)$ which preserves $\xi$; hence it extends to a linear isometry of $X^1(\R^N)$, which coincides with $v \mapsto v \circ \tau$ by the continuity of the embedding $X^1(\R^N) \subset L^2(\R^N)$. Therefore $Qv := \frac12 (v - v\circ\tau)$ is a bounded projection of $X^1(\R^N)$ onto $X_\tau$, which maps $\cC_0^\infty(\R^N)$ into $\cC_0^\infty(\R^N) \cap X_\tau$. Let $u \in X_\tau$, take $\varphi_n \in \cC_0^\infty(\R^N)$ with $\varphi_n \to u$ in $X^1(\R^N)$ and put $u_n := Q\varphi_n$. Then $u_n \to Qu = u$ in $X^1(\R^N)$ and, by the first part applied to $u_n - u_m$,
$$
\|\nabla(u_n - u_m)\|_{L^2(\R^N)}^2 \leq \frac{N^2}{4(N-1)}\, \xi(u_n - u_m) \to 0 \quad \mbox{as } n,m \to \infty .
$$
Thus $(u_n)$ is a Cauchy sequence in $H^1(\R^N)$; its limit in $H^1(\R^N)$ coincides with its limit $u$ in $L^2(\R^N)$. Hence $u \in H^1(\R^N)$, and \eqref{eq:Xtau-coercive} follows by passing to the limit. The remaining claims are immediate consequences of \eqref{eq:Xtau-coercive}.
\end{proof}

Recall that we assume $g$ is odd and we denote $H^1_{\mathcal{O}}(\R^N) := X^1_{\mathcal{O}}(\R^N) \cap H^1(\R^N)$. Now, we can repeat the proof of Theorem \ref{Th:splitting}, using \cite[Corollary 3.2]{MR4173560} and Lemma \ref{L:Xtau}, and obtain the following theorem.

\begin{Th}\label{th:dekompNonradial}
Suppose that $(u_n) \subset X_\tau \cap H^1_{\mathcal{O}} (\R^N)$ is a bounded in $X^1(\R^N)$, Poho\v{z}aev-Palais-Smale sequence for $J$ at level $c > 0$. Then, passing to a subsequence, there are an integer $k \geq 0$, $u_0 \in X_\tau \cap X^1_{\mathcal{O}} (\R^N)$, and sequences $(y_n^j) \subset \R^{N-2M}$, $w^j \in X_\tau \cap H^1_{\mathcal{O}} (\R^N)$ for $j\in \{1,\ldots, k\}$ such that

\begin{itemize}
\item[(a)] $u_n \weakto u_0$, $J'(u_0)=0$ and $P(u_0)=0$;
\item[(b)] $|y_n^j| \to \infty$ and $|y_n^j - y_n^{j'}|\to \infty$ for $j \neq j'$;
\item[(c)] $w^j \neq 0$ and $w^j$ is a weak solution to
$$
-\Delta w^j = g(w^j) \quad \mbox{in } \R^N
$$
and is a critical point of $J_{\infty}$ for each $1 \leq j \leq k$;
\item[(d)] $u_n - u_0 - \sum_{j=1}^k w^j (\cdot - (0,0,y_n^j)) \to 0$ in $X^1 (\R^N)$;
\item[(e)] $\|u_n\|^2_\ell\to \|u_0\|^2_\ell+\sum_{j=1}^k  \left( \|\nabla w^j\|_{L^2 (\R^N)}^2 + \ell \| w^j \|_{L^2 (\R^N)}^2 \right)$;
\item[(f)] $J(u_n) \to J(u_0) + \sum_{j=1}^k J_\infty (w^j)$.
\end{itemize}
\end{Th}

The only difference in the proof is that we know $w^j$ are weak solutions to the limiting problem in $X_\tau \cap H^1_\mathcal{O} (\R^N)$, so a critical point of $J_\infty |_{X_\tau \cap H^1_{\mathcal{O}}}$. However, from Palais' principle of symmetric criticality, it is a critical point of the free functional $J_\infty$. Similarly we know that $u_0$ is a critical point of $J$.

In the same way as in Lemma \ref{r}, there is $\rho > 0$ such that
$$
\inf_{\|u\|_\ell = \rho} J(u) > 0.
$$
To show that $J$ has a mountain pass geometry, in Lemma \ref{v} it is shown how to, for any $R > 0$, construct a radial function $w_R \in H^1(B(0,R+1)) \cap L^\infty (\R^N) \subset X^1 (\R^N)$ such that 
$$
\int_{\R^N} G(w_R) \, dx \geq C_1 R^N - C_2 R^{N-1}, \quad C_1, C_2 > 0,
$$
and
$$
\|w_R\|_{L^2 (\R^N)}^2 \geq C_3 R^N, \quad C_3 > 0.
$$
Let $R > 1$. Now, as in \cite[Remark 4.2]{MR4173560}, we choose an odd and smooth function $\varphi : \R \rightarrow [-1,1]$ such that $\varphi(x) = 1$ for $x \geq 1$, and $\varphi(x) = -1$ for $x \leq -1$. Then, as shown in \cite[Remark 4.2]{MR4173560}, the function (with a small abuse of notation $w_R (x) = w_R( |x| )$) 
$$
\widetilde{w}_R (x_1, x_2, x_3) := w_R ( \sqrt{|x_1|^2+ |x_2|^2+ |x_3|^2}) \varphi(|x_1|-|x_2|)
$$
belongs to $X_\tau \cap H^1_{\mathcal{O}} (\R^N)$ and
\begin{align*}
\int_{\R^N} G(\widetilde{w}_R) \, dx &\geq C_4 R^N - C_5 R^{N-1}, \\
\int_{\R^N} | \widetilde{w}_R |^2 \, dx &\geq C_6 R^N - C_7 R^{N-1}, \\
\int_{\R^N} | \nabla \widetilde{w}_R |^2 \, dx &\leq C_8 R^{N-1}.
\end{align*}
Hence, for sufficiently large $R > 1$ we have 
$$
J(\widetilde{w}_R) < 0, \quad \|\widetilde{w}_R\|_\ell > \rho.
$$

\begin{proof}[Proof of Theorem \ref{Th:Main2}]

From the above reasoning, $J$ has the mountain pass geometry, and repeating the same arguments as in Lemma \ref{Lem:ExistencePPS} we find a Poho\v{z}aev-Palais-Smale sequence for $J$ on the mountain pass level. Then, using Theorem \ref{th:dekompNonradial} repeating the proof of Theorem \ref{Th:Main1-Existence}(a), we find a nontrivial critical point of $J |_{X_\tau \cap X^1_{\mathcal{O}} (\R^N)}$, which - thanks to Palais' principle of symmetric criticality - is also a critical point of $J$. Moreover the solution belongs to $\cM$, and - as in Theorem \ref{Th:Main1-Existence}(b) - it is a minimizer of $J$ on $\cM \cap X_\tau \cap X^1_{\mathcal{O}} (\R^N)$.

Let $v_0 \in \cM \cap X_\tau \cap X^1_{\mathcal{O}} (\R^N)$ be the minimizer found above. It remains to show
$$
J(v_0) \geq 2 \inf_\cM J.
$$
Define two open sets

$$
\Omega_1 := \{ x \in \R^N \ : \ |x_1| > |x_2| \}, \quad \Omega_2 := \{ x \in \R^N \ : \ |x_1| < |x_2| \}.
$$
Define $u_1 := \mathbbm{1}_{\Omega_1} v_0$, $u_2 := \mathbbm{1}_{\Omega_2} v_0$. Then $u_1, u_2 \in H^1 (\R^N)$, since $v_0 = 0$ on $\{|x_1|=|x_2|\}$, and $\xi(v_0) = 2 \xi(u_1) = 2 \xi(u_2)$. We also have $\int_{\R^N} G(v_0) \, dx = 2 \int_{\R^N} G(u_1) \, dx$. Hence $u_1 \in \cM$, $J(u_1) = J(u_2)$ and since $u_1$, $u_2$ have disjoint supports, we get
$$
J(v_0) = J(u_1) + J(u_2) = 2 J(u_1) \geq 2 \inf_{\cM} J.
$$
\end{proof}

\section*{Acknowledgments}
Bartosz Bieganowski and Daniel Strzelecki were partly supported by the National Science Centre, Poland (Grant No. 2022/47/D/ST1/00487).

\textbf{Competing interests.} The authors have no competing interests to declare that are relevant to the content of this article.

\textbf{Data availability.} Not applicable.

\textbf{AI statement.} The authors used AI tools (Anthropic Claude) to assist with manuscript editing and final verification of the mathematical arguments developed by the authors. The authors take full responsibility for the mathematical correctness and final content of the
manuscript.

\bibliographystyle{acm}
\bibliography{Bibliography}

\end{document}